\documentclass{amsart}

\usepackage{amssymb,amsmath,amsthm}
\usepackage{mathrsfs}
\usepackage{hyperref}
\usepackage{mathtools}
\usepackage{color}
\usepackage{verbatim}
\usepackage{tikz-cd}

\mathtoolsset{showonlyrefs}

\renewcommand{\[}{\begin{equation}\begin{aligned}}
\renewcommand{\]}{\end{aligned} \end{equation}}

\renewcommand{\P}{\mathbb{P}}
\newcommand{\R}{\mathbb{R}}
\newcommand{\C}{\mathbb{C}}
\newcommand{\RP}{\mathbb{RP}}

\newcommand{\fix}{\operatorname{Fix}}

\usepackage{todonotes}

\newtheorem{thm}{Theorem}
\newtheorem{prop}[thm]{Proposition}
\newtheorem{lemma}[thm]{Lemma}

\newtheorem{cor}[thm]{Corollary}

\theoremstyle{remark}
\newtheorem{remark}[thm]{Remark}
\newtheorem{ex}[thm]{Example}

\theoremstyle{definition}
\newtheorem{definition}[thm]{Definition}
\newtheorem{assumption}[thm]{Assumption}

\numberwithin{equation}{section}
\numberwithin{thm}{section}

\title[Nowhere vanishing harmonic $1$-forms]{Nowhere-vanishing
  harmonic 1-forms on real loci of K3-fibred Calabi-Yau 3-folds}

\author{Shih-Kai Chiu}
\address{Department of Mathematics, University of California, Irvine, Irvine CA 92697, USA}
\email{shihkaic@uci.edu}

\author{Daniel Platt}
\address{Department of Mathematics, Imperial College London, 180 Queen’s Gate, South Kensington, London SW7 2RH, UK}
\email{d.platt@imperial.ac.uk}

\author{Calum Spicer}
\address{Department of Mathematics, King’s College London, Strand, London WC2R 2LS, UK}
\email{calum.spicer@kcl.ac.uk}

\begin{document}

\begin{abstract}
  We develop an analytic construction of nowhere-vanishing harmonic
  $1$-forms on real loci of K3-fibred Calabi-Yau $3$-folds with
  collapsing Ricci-flat K\"ahler metrics. We apply our construction to
  examples whose real loci have connected components diffeomorphic to
  $S^1\times S^2$ and to both trivial and nontrivial mapping tori. As
  an application, we produce examples of compact $7$-manifold with
  holonomy $G_2$ via the Joyce-Karigiannis construction.
\end{abstract}

\maketitle

\section{Introduction}

In 1996, Joyce~\cite{JoyceI, JoyceII} constructed the first examples
of compact Riemannian $7$-manifolds with holonomy equal to
$G_2$. Since then, constructing new examples of compact manifolds with
holonomy $G_2$ has been one of the central challenges in
$G_2$-geometry. Joyce’s first examples were obtained by a generalised
Kummer construction, in analogy with the Kummer construction of
Ricci-flat K\"ahler metrics on K3 manifolds
\cite{Topiwala,DonaldsonKummer}. In 2003, Kovalev \cite{Kovalev}
constructed new examples via the twisted connected sum construction;
see also \cite{KovalevLee,CHNP, Nordstrom} for more examples and
generalisations. In 2021, Joyce and Karigiannis
\cite{JoyceKarigiannis} introduced a new construction of compact $G_2$
manifolds, which produces many new examples; see
Section~\ref{sec:preliminaries} for an overview of the results about
this construction.

In \cite[Section 7.5]{JoyceKarigiannis}, Joyce and Karigiannis
describe a conjectural application of their construction, which we now
explain. Let $X$ be a Calabi-Yau $3$-fold with an anti-holomorphic
involution $\sigma: X \to X$. Let $\omega$ be the unique Ricci-flat
K\"ahler metric in a K\"ahler class that is anti-invariant under the
involution, namely $\sigma^*[\omega] = -[\omega]$. It follows from
uniqueness of the Ricci-flat K\"ahler metric that the involution is
also anti-symplectic, that is, $\sigma^*\omega = -\omega$.

Let $\Omega$ be a holomorphic volume form on $X$. After multiplying
$\Omega$ by a unit complex number, we may assume that
$\sigma^*\Omega = \bar\Omega$. Let us further assume that the fixed
locus $L$ is nonempty. It follows that each connected component of $L$
is a special Lagrangian submanifold.

Now, let us consider the product $M = S^1 \times X$ equipped with the
torsion-free $G_2$ structure
$\varphi = d\theta\wedge\omega + \operatorname{Re}\Omega$. There is a
natural $\mathbb{Z}_2$-action generated by
$(e^{i\theta}, x) \mapsto (e^{-i\theta}, \sigma(x))$. The $3$-form
$\varphi$ is invariant under the $\mathbb{Z}_2$-action, and hence
descends to a torsion-free $G_2$-structure on the quotient
$M/\mathbb{Z}_2$, whose singular set is diffeomorphic to two copies of
$L$. Provided $L$ admits a nowhere-vanishing harmonic $1$-form, the
Joyce-Karigiannis construction resolves the singularities by gluing in
Eguchi-Hanson spaces according to this harmonic $1$-form; see
Section~\ref{sec:preliminaries} for more details. Thus, in this
setting, the main difficulty is to prove the existence of such a
$1$-form on $L$. As explained in \cite[Section~7.5]{JoyceKarigiannis},
this is a difficult problem, since the induced metric on $L$ is in
general not explicit enough.

The main purpose of this paper is to demonstrate that there is a large
class of examples via collapsing the Ricci-flat K\"ahler metrics on
Calabi-Yau $3$-folds fibred by K3 surfaces.  Our main theorem is the
following:



\begin{thm}
  \label{thm:main1}
  Let $X$ be a Calabi-Yau $3$-fold and let $Y$ denote the complex
  projective line $\mathbb{CP}^1$ equipped with its Fubini-Study
  metric $\omega_Y$. Suppose $\pi: X \to Y$ is a Lefschetz K3
  fibration and $\sigma: X \to X$ is an anti-holomorphic involution
  such that $\pi$ is equivariant with respect to $\sigma$ and complex
  conjugation on $Y$. Assume that the fixed locus $L$ is nonempty and
  contains no critical point of $\pi$. Let $\omega_X$ be a K\"ahler
  metric on $X$ such that $\sigma^*[\omega_X] = -[\omega_X]$.  For
  $t>0$, let $\tilde\omega_t$ be the unique Ricci-flat K\"ahler metric
  in the K\"ahler class $[\omega_X] + (1/t)\pi^*[\omega_Y]$, and let
  $\tilde g_t$ be the associated Riemannian metric. Then for all
  sufficiently small $t>0$, each connected component of $L$ admits a
  nowhere-vanishing harmonic $1$-form with respect to the induced
  metric $\tilde g_t|_L$.
\end{thm}

\begin{remark} In this paper, Calabi-Yau $n$-folds are
    assumed to be compact and simply connected; in particular,
    Calabi-Yau $3$-folds have holonomy $SU(3)$ and satisfy
    $h^{1,0} = h^{2,0} = 0$.
\end{remark}

\begin{remark}
  In Theorem~\ref{thm:main1}, the collapsing Ricci-flat K\"ahler
  metric is $t\,\tilde\omega_t$. The family $(X,t\,\tilde g_t)$
  converges in the Gromov-Hausdorff sense to $Y$ equipped with the
  generalized K\"ahler-Einstein metric $\tilde\omega_Y$, whose Ricci
  curvature is the Weil-Peterson metric of the fibration. We refer the
  reader to \cite{Tosatti} for an overview about this convergence
  result.
\end{remark}

\begin{remark}
  To our knowledge, this is the first analytic construction of
  nowhere-vanishing harmonic $1$-forms on real loci of Calabi-Yau
  $3$-folds; compare the numerical evidence in
  \cite{DouglasPlattQiBarbosa}.
\end{remark}

We briefly explain the idea of the proof. For simplicity, let us
assume $L$ has one connected component. Then $L$ fibres
over $S^1$, the real locus of $\mathbb{CP}^1$, and each
fibre is a special Lagrangian submanifold in the corresponding K3
fibre. The idea is to pull back a suitable nowhere-vanishing $1$-form
$\alpha$ defined on $S^1$ and then correct it by an exact $1$-form
$du$ to obtain a harmonic $1$-form.

The main point is to choose the $1$-form $\alpha$ carefully. Recent
regularity results for collapsing Ricci-flat K\"ahler metrics on
K3-fibered Calabi-Yau $3$-folds \cite{Li, HeinTosatti, ChiuLin},
together with a related construction of special Lagrangian
submanifolds~\cite{ChiuLin}, imply that locally the induced metric
$\tilde g_t|_L$ is asymptotic to a product metric as $t\to 0^+$;
globally, however, $\tilde g_t|_L$ is not approximately a product. To
account for this, we choose $\alpha$ whose coefficient with respect to
the arc length parameter is roughly given by the inverse of the area
of the corresponding special Lagrangian fibre. The effect is that the
fibrewise average of $d^*\alpha$ is much smaller than the size of
$d^*\alpha$ itself.

This cancellation property enables us to solve $\Delta u = d^*\alpha$
with uniform estimates for all sufficiently small $t>0$, so that $du$
is much smaller than $\alpha$. Although the Hodge decomposition on $L$
gives the correction term $du$ for each $t$, it does not provide the
uniform estimates needed to control $du$ as $t \to 0^+$. Moreover, it
does not explain the crucial choice of $\alpha$. Instead, we solve the
equation by a parametrix construction adapted to the geometry of
$\tilde g_t|_L$. The perturbation $\alpha-du$ then gives the desired
nowhere-vanishing harmonic $1$-form for sufficiently small $t>0$. We
refer the reader to Section~\ref{sec:analysis} for the proof of
Theorem~\ref{thm:main1}.

We apply Theorem~\ref{thm:main1} to find examples of Calabi-Yau
$3$-folds whose real loci support nowhere-vanishing harmonic
$1$-forms.

\begin{thm}
  \label{thm:main2}
  There exist K3-fibred Calabi-Yau $3$-folds with anti-holomorphic
  involutions whose real loci support nowhere-vanishing harmonic
  $1$-forms. These examples include complete intersections, bidegree
  $(2,4)$ hypersurfaces of $\mathbb{P}^1 \times \mathbb{P}^3$,
  hypersurfaces in scrolls, and double covers. In particular, the real
  loci obtained this way include connected components diffeomorphic to
  $S^1 \times S^2$, $S^1 \times \Sigma_g$ and nontrivial mapping tori
  with finite monodromy.
\end{thm}

An immediate corollary of Theorem~\ref{thm:main2} is the existence of
families of pairwise disjoint special Lagrangian submanifolds:

\begin{cor}
  There exist $1$-parameter families of pairwise
  disjoint special Lagrangian submanifolds in certain K3-fibred
  Calabi-Yau 3-folds. Examples include $S^1\times S^2$,
  $S^1 \times \Sigma_g$ and nontrivial mapping tori of finite
  monodromy.
\end{cor}

\begin{remark}
  This is a direct application of a quantitative version of McLean's
  theorem~\cite{McLean}. Such a phenomenon is already familiar for
  special Lagrangian submanifolds diffeomorphic to $T^3$; see
  \cite{LiSYZ}. To our knowledge, however, these are the
  first known examples of local families of pairwise disjoint special
  Lagrangian submanifolds diffeomorphic to $S^1 \times S^2$.
\end{remark}

Using the Joyce-Karigiannis construction, we produce examples of
compact, full holonomy $G_2$-manifolds from Lefschetz K3 fibrations
with nonempty real loci.

\begin{cor}
  \label{cor:main2}
  There exist compact manifolds with holonomy $G_2$ obtained by
  applying the Joyce-Karigiannis construction to
  the product of a circle with K3-fibred Calabi-Yau $3$-folds.  In particular, one
  can construct simply-connected $G_2$ manifolds with
  \[
    (b^2, b^3)=(2,95),\,(2,91),\,(20,109),\,(14,111),\,(2,91),\,(2,133).
  \]
\end{cor}

To our knowledge, $G_2$-manifolds with
$(b^2,b^3) = (20,109),(14,111),(2,133)$ do not appear in the published
lists of known examples (namely \cite[Table 1]{KovalevLee},
\cite[Theorems 4.2, 4.6, 4.11]{reidegeld}, \cite[Table 2]{JoyceII}),
so our examples may be new.

We construct the examples above by choosing suitable K3 fibrations
with the compatible real structures. The key input is the generic
smoothness of these K3 fibrations; see
Proposition~\ref{prop_lefschetz}. For each example, we compute the
Betti numbers of the $G_2$-manifold using the method of
Joyce-Karigiannis~\cite{JoyceKarigiannis}. We refer the reader to
Section~\ref{sec:preliminaries} for the method and to
Section~\ref{sec:examples} for details about these examples. It would
be interesting to find further examples and study further topological
invariants of the associated $G_2$-manifolds.

The construction in this paper may be viewed as a source of examples
for several related programs. Our examples of Calabi-Yau $3$-folds
with Lefschetz fibrations and compatible real structures provide
inputs for the Joyce-Karigiannis construction. The resulting
$G_2$-manifolds may be regarded as test cases for Li's iterated
collapsing picture of compact $G_2$-manifolds \cite{LiG2}, which
builds on the work of Apostolov-Salamon on $S^1$-invariant $G_2$
structures \cite{ApostolovSalamon}. On the other hand, it is also
natural to ask whether these examples can shed light on possible
relations between the Joyce-Karigiannis construction and Kovalev's
twisted connected sum construction. More broadly, these new
$G_2$-manifolds may serve as testing grounds for the further study of
gauge theory and conjectural enumerative theories of calibrated
submanifolds in manifolds with special holonomy \cite{DonaldsonThomas,
  DonaldsonSegal, DonaldsonScaduto}.

This paper is organised as follows. In Section~\ref{sec:preliminaries}
we review results of real loci of Calabi-Yau manifolds and the
Joyce-Karigiannis construction. In Section~\ref{sec:analysis} we prove
Theorem~\ref{thm:main1}. In Section~\ref{sec:k3}, we prove the generic
smoothness result for fibrations, Proposition~\ref{prop_lefschetz},
which will be used in the construction of K3 fibrations.  Finally, in
Section~\ref{sec:examples} we provide the examples mentioned in
Theorem~\ref{thm:main2} and Corollary~\ref{cor:main2}, and verify
their topological properties. \newline

\noindent{\bf Acknowledgements.}
We are grateful to Simon Donaldson for drawing our attention to
collapsing Ricci-flat K\"ahler metrics on Calabi-Yau 3-folds fibred
by K3 surfaces, which provided much of the motivation for the results
of this paper. Part of this work was done while SC was visiting Julie
Tzu-Yueh Wang at Academia Sinica, he is grateful for her warm
hospitality. CS was partially supported by EPSRC.

\section{Preliminaries}
\label{sec:preliminaries}

\subsection{Real loci of Calabi-Yau manifolds}

\begin{definition}
  $(X,\omega,\Omega)$ is a Calabi-Yau $n$-fold if $(X,\omega)$ is a
  simply connected, compact K\"ahler manifold of complex
  dimension $n$ and $\Omega$ is a nowhere-vanishing holomorphic
  section of the canonical bundle $K_X$, called a holomorphic volume
  form on $X$, such that $\omega$ and $\Omega$ solve the complex
  Monge-Amp\`ere equation
  \[\label{eq:normalization}
    \frac{\omega^n}{n!}  =
    (-1)^{n(n-1)/2}\left(\frac{\sqrt{-1}}{2}\right)^n\Omega\wedge \bar{\Omega}.
  \]
\end{definition}

By Yau's theorem~\cite{Yau}, on a compact K\"ahler manifold $X$ with a
holomorphic volume form $\Omega$, every K\"ahler class admits a unique
K\"ahler metric $\omega$ solving \eqref{eq:normalization}. Note that
from \eqref{eq:normalization}, $\omega$ is Ricci flat.

Let $(X,\omega,\Omega)$ be a Calabi-Yau $n$-fold. Let us assume that
$X$ is equipped with an \emph{anti-holomorphic involution}
$\sigma: X \to X$, which by definition is a diffeomorphism from $X$ to
$X$ such that $\sigma^2 = \operatorname{Id}$ and
$J\circ \sigma_* = - \sigma_* \circ J$, where $J$ is the complex
structure of $X$. By multiplying $\Omega$ with a unit complex number,
we may assume that $\sigma^*\Omega = -\bar\Omega$. In addition, we
assume that the K\"ahler class of the Ricci-flat K\"ahler metric
$\omega$ is anti-invariant under $\sigma$, that is,
$\sigma^*[\sigma] = -[\sigma]$. Since $[\omega] = [-\sigma^*\omega]$,
by uniqueness we have $\sigma^*\omega = -\omega$. In other words,
$\sigma$ is an \emph{anti-symplectic involution}.

We are interested in the fixed locus $\operatorname{Fix}(\sigma)$ of
an anti-holomorphic, anti-symplectic involution $\sigma: X \to X$. If
$\operatorname{Fix}(\sigma)$ is nonempty, then it is a smooth real
$n$-dimensional manifold. To see this, note that near each fixed
point, one can choose a suitable local holomorphic chart in which
$\sigma$ is given by complex conjugation. Thus, locally
$\operatorname{Fix}(\sigma)$ can be identified with
$\mathbb{R}^n \subset \mathbb{C}^n$, and the transition functions on
$X$ restrict to smooth transition functions on
$\operatorname{Fix}(\sigma)$. Moreover, $\operatorname{Fix}(\sigma)$
is a \emph{special Lagrangian submanifold}:

\begin{definition}
  \label{def:SL}
  An oriented Lagrangian submanifold $L$ in a Calabi-Yau $n$-fold
  $(X,\omega,\Omega)$ is \emph{special Lagrangian} if
  $\Omega|_L=e^{i\theta} \operatorname{vol}_L$ for some constant
  $\theta \in \mathbb{R}$. Equivalently, this means that
  $\operatorname{Im}(e^{-i\theta}\Omega)|_L = 0$ and $L$ is calibrated
  by the calibration form $\operatorname{Re}(e^{-i\theta}\Omega)$ in
  the sense of Harvey-Lawson~\cite{HarveyLawson}.
\end{definition}

\begin{remark}
  In general, for an oriented Lagrangian submanifold $L$, we always
  have $\Omega|_L=e^{i\theta_L} \operatorname{vol}_L$, where
  $e^{i\theta_L}$ is an $S^1$-valued function on $L$. This can be seen
  from \eqref{eq:normalization}. The mean curvature vector of $L$ is
  given by $\mathbf{H} = J\nabla \theta_L$.
\end{remark}

\begin{lemma}
  Let $(X,\omega,\Omega)$ be a Calabi-Yau $n$-fold with an
  anti-holomorphic, anti-symplectic involution $\sigma$. If
  $\operatorname{Fix}(\sigma)$ is nonempty, then it is a special
  Lagrangian submanifold.
\end{lemma}

\begin{proof}
  Let $L = \operatorname{Fix}(\sigma)$. Since
  $\omega|_L = -\sigma^*\omega|_L = -\omega|_L$, it follows that
  $\omega|_L = 0$ and $L$ is Lagrangian. Since
  $\operatorname{Im}\Omega|_L = -\sigma^*\operatorname{Im}\Omega|_L =
  -\operatorname{Im}\Omega|_L$, it follows that
  $\operatorname{Im}\Omega|_L = 0$, so $L$ is special Lagrangian.
\end{proof}

It follows from the lemma that special Lagrangian submanifolds can be
constructed as real loci of real Calabi-Yau manifolds. In order to
study the geometry of the real loci more explicitly, one needs to
understand the metric behaviour of the Calabi-Yau manifold and its real
locus. To this end, we will focus on the case when the Calabi-Yau
$3$-fold admits a Lefschetz K3 fibration.

\begin{definition}
\label{def:lefschetz-fibration}
Let $X$ be a complex $n$-dimensional manifold and let $Y$ be a complex
$1$-dimensional manifold. A holomorphic surjection $\pi: X\to Y$ with
connected fibres is called a Lefschetz fibration if, for each critical
point $p \in X$, there exist local holomorphic coordinates
$(z_1,\ldots,z_n)$ centred at $p$ and a local holomorphic coordinate
$y$ centred at $\pi(p)$ such that
$y\circ \pi = z_1^2 + \cdots + z_n^2$. For each $y \in Y$, we write
$X_y = \pi^{-1}(y)$ for the corresponding fibre. We write $S$ for the
set of critical values of $\pi$.
\end{definition}

Now we assume $X$ is a Calabi-Yau $3$-fold that admits a Lefschetz
fibration $\pi: X \to Y$. We claim that
  $\pi_*: \pi_1(X) \to \pi_1(Y)$ is surjective, so $Y$ must be
  $\mathbb{CP}^1$. To see this, take a loop $\gamma: [0,1] \to
  Y$. Since $S$ is finite, by a small perturbation we may assume that
  $\gamma$ lies in $Y^\circ = Y \setminus S$. Let
  $X^\circ = \pi^{-1}(Y^\circ)$. Let $\tilde\gamma$ be a lift of
  $\gamma$ to $X^\circ$. Note that the endpoints of $\tilde\gamma$ lie
  in the same fibre $X_{\gamma(0)}$. Since $X_{\gamma(0)}$ is
  connected, the endpoints of $\tilde\gamma$ can be joined by a path
  contained in $X_{\gamma(0)}$. As a result, we obtain a loop whose
  image under $\pi_*$ is homotopic to $\gamma$. This shows that
  $\pi_*: \pi_1(X) \to \pi_1(Y)$ is surjective.

By adjunction, $K_{X_y}$ is trivial for $y \in Y$. Thus $X_y$ is
either a K3 surface or an abelian surface. The latter case is excluded
by the Lefschetz condition: otherwise, each regular fibre near a
singular point must contain a vanishing cycle diffeomorphic to
$S^2$. This $S^2$ has self-intersection $-2$. On the other hand, since
an abelian surface is diffeomorphic to $T^4$, we have
$\pi_2(T^4) = 0$, and hence this $S^2$ has self-intersection $0$, a
contradiction.

\subsection{$G_2$-manifolds and the Joyce-Karigiannis construction}
A $G_2$-structure on a $7$-manifold $M$ is given by a positive
$3$-form $\varphi\in\Omega^3(M)$, which by definition means that, at
each point $p \in M$, $\varphi$ is identified with the standard
$G_2$-invariant $3$-form
\[
  \varphi_0
  =
  e^{123}+e^{145}+e^{167}
  +e^{246}-e^{257}-e^{347}-e^{356}
\]
on $\mathbb R^7$ by choosing a basis $e^1,\ldots,e^7$ of
$T^*_pM$. $\varphi$ determines a Riemannian metric $g_\varphi$, an
orientation, and hence a Hodge star $\star_\varphi$. The
$G_2$-structure $\varphi$ is called torsion-free if
\[
  d\varphi=0,\qquad d\star_\varphi\varphi=0.
\]
Equivalently, $\varphi$ is parallel with respect to the Levi-Civita
connection of $g_\varphi$, and in this case
$\operatorname{Hol}(g_\varphi)\subseteq G_2$. We call $(M,\varphi)$ a
$G_2$-manifold. If the holonomy is equal to $G_2$, we say that
$(M,\varphi)$ has full $G_2$-holonomy. A comprehensive introduction to
the topic is \cite[Section 10]{Joycebook}.

A central challenge in $G_2$-geometry is to construct compact
$7$-manifolds with full $G_2$ holonomy. In the following, we briefly recall the construction from \cite{JoyceKarigiannis}.
Let $(M, \varphi)$ be a compact, torsion-free $G_2$-manifold and let $g$ denote the metric induced by $\varphi$.
Let $\iota: M \rightarrow M$ be a non-trivial involution so that $\iota^* \varphi=\varphi$.
Then $L'=\fix (\iota)$ is a (possibly empty) compact associative $3$-fold by \cite[Prop. 10.8.1]{Joycebook}.
The main result of \cite{JoyceKarigiannis} is then:

\begin{thm}[Theorem 6.4 in \cite{JoyceKarigiannis}]
\label{thm:joyce-karigiannis}
    Suppose $L'$ is nonempty, and that there exists a closed, coclosed, nowhere-vanishing $1$-form $\lambda \in \Omega^1(L')$, i.e. $d \lambda=d^* \lambda=0$, where $d^*$ is defined using $g|_{L'}$.

    Then there exists a compact $7$-manifold $N$ defined as a resolution of singularities $\pi:N \rightarrow M/\langle \iota \rangle$ of the $7$-orbifold $M/\langle \iota \rangle$ along its singular locus $L'$ by gluing in a bundle $P \rightarrow L'$ along $L'$ with fibre the Eguchi-Hanson space.
    The fundamental group of $N$ satisfies $\pi_1(N)=\pi_1(M / \langle \iota \rangle)$ and the Betti numbers are
    \begin{align}
    \label{eqn:jk-betti-formula}
        b^k(N)
        =
        b^k(M/\langle \iota \rangle)+b^{k-2}(L).
    \end{align}
    There exists a smooth family $\tilde \varphi^N_t$ of torsion-free $G_2$-structures on $N$ for $t \in (0,\epsilon]$, with $\epsilon > 0$ small, such that $\tilde \varphi^N_t \rightarrow \pi^* \varphi$ in $C^0$ away from $Q$ as $t \rightarrow 0$, and for each $x \in L'$ the fibre $\pi^{-1}(x) \cong S^2$ with metric induced by $\tilde \varphi^N_t$ approximates a small round $2$-sphere with area $\pi t^2 |\lambda_x|$ for small $t$.
    The metric induced by $\tilde \varphi^N_t$ has holonomy $G_2$ if and only if $M/\langle \iota \rangle$ has finite fundamental group.
\end{thm}

There are two situations where this crucial ingredient is known to exist:
first, the original generalised Kummer construction from \cite{JoyceI, JoyceII} is recovered as a special case of the gluing theorem from \cite{JoyceKarigiannis}, where one starts with $M=T^7$ and the flat $G_2$-structure.
Second, one may begin with $M=T^3 \times \text{K3}$.
Two examples of this type are explained in \cite[Section 7.2]{JoyceKarigiannis}, and a full classification was obtained in \cite{reidegeld}.
In these cases, the nowhere vanishing harmonic $1$-form comes from the flat $T^7$ and $T^3$, respectively.
Even though the initial manifold has reduced holonomy, many of the resulting examples have full holonomy $G_2$.

In this paper, we apply the theorem to products of Calabi-Yau manifolds with a circle, as done in the conjectural example \cite[Section 7.5]{JoyceKarigiannis}.
Let $(X, \omega, \Omega)$ be a Calabi-Yau manifold.
Let $\sigma : X \rightarrow X$ be an anti-holomorphic involution so that $\sigma^* \omega=-\omega$ and $\sigma^* \Omega=\bar \Omega$.
Extend $\sigma$ to $M=S^1 \times X$ as $\iota: M \rightarrow M$ as $\iota(e^{i\theta},x)=(e^{-i\theta},\sigma(x))$.
Then the $G_2$-structure $\varphi=d\theta \wedge \omega+\text{Re} \, \Omega \in \Omega^3(M)$ is preserved by $\iota$, so Theorem \ref{thm:joyce-karigiannis} can be applied to $M$.
We obtain:

\begin{cor}
    \label{cor:joyce-karigiannis-for-cy}
    Let $(X, \omega, \Omega)$ be a Calabi-Yau manifold and $\sigma, \iota$ as above and $L=\fix (\sigma)$.
    Then there exists a resolution of singularities $N \rightarrow M/\langle \iota \rangle$ carrying a smooth family $\tilde \varphi^N_t$ of torsion-free $G_2$-structures.
    The induced metrics have holonomy equal to $G_2$ and the Betti numbers of $N$ satisfy
    \begin{align}
    \begin{split}
    b^0(N)=1, \quad
    b^1(N)=0, \quad
    b^2(N)=\dim (H^{1,1}(X)^\sigma)+2b^0(L),
    \\
    b^3(N)=1+h^{2,1}(X)+\dim (H^{1,1}(X)^{-\sigma})+2b^1(L),
    \end{split}
    \end{align}
    where $H^{1,1}(X)^{\pm\sigma}$ denote the $\pm 1$-eigenspaces of
    $\sigma^*$, respectively.
\end{cor}

\begin{proof}
    By \eqref{eqn:jk-betti-formula}, we have
    $b^0(N)=b^0(M/\langle \iota \rangle)=1$, because $X$ was assumed to be connected.
    By the same equation,
    $b^1(N)=b^1(M/\langle \iota \rangle)=\dim H^1(M)^\sigma=b^1(X)^\sigma+b^1(S^1)^{-\sigma}=0$, by the Künneth formula, where the last equality holds because $X$ is simply connected by our definition of Calabi-Yau.
    The equation for $b^2(N)$ follows because $h^{2,0}(X)=0$ for a Calabi-Yau $3$-fold.
    By the K\"unneth formula, $H^3(S^1 \times X) \cong H^3(X) + H^2(X) \otimes H^1(S^1)$.
    Taking the $\iota$-invariant part gives
    \[
    H^3((S^1 \times X)/\langle \iota \rangle)
    =
    H^3(S^1 \times X)^\iota
    \cong
    H^3(X)^\sigma
    +
    H^2(X)^{-\sigma}.
    \]
    The anti-holomorphic involution $\sigma$ acts on $H^3(X)$ by swapping $H^{p,q}(X)$ and $H^{q,p}(X)$, hence it preserves $\Omega + \bar \Omega \in H^{3,0}(X) \oplus H^{0,3}(X)$ and a copy of $H^{2,1}(X)$ in $H^{2,1}(X) \oplus H^{1,2}(X)$, so $\dim (H^3(X)^\sigma)=1+h^{2,1}(X)$.
    The equation for $b^3(N)$ follows then from \eqref{eqn:jk-betti-formula}.
\end{proof}

\section{Nowhere vanishing harmonic $1$-forms}
\label{sec:analysis}

In this section, we prove Theorem~\ref{thm:main1}. We reformulate
Theorem~\ref{thm:main1} as follows.

\subsection{Setup}
Let $X$ be a Calabi-Yau $3$-fold with a Lefschetz K3 fibration
\[
  \pi: X \to Y=\mathbb{P}^1.
\]
Let $S \subset Y$ denote the discriminant locus of $\pi$. 

\begin{assumption}\label{assumption:1}
  We further assume the following:
  \begin{itemize}
  \item There exist anti-holomorphic involutions $\sigma: X \to X$ and
    $\sigma: Y \to Y$.
  \item The fibration is equivariant with respect to $\sigma$, i.e.,
    $\pi\circ\sigma = \sigma\circ\pi$.
  \item The fixed locus $\operatorname{Fix}(\sigma) \subset X$ is
    nonempty.
  \item The fixed locus $\gamma=\operatorname{Fix}(\sigma) \subset Y$
    contains no points in the discriminant locus $S$.
  \end{itemize}  
\end{assumption}

Let us fix a holomorphic volume form $\Omega$ of $X$ and a reference
K\"ahler metric $\omega_X$ of $X$ such that
$\sigma^*[\omega_X] = -[\omega_X]$. Let $\omega_Y$ denote the
Fubini-Study metric on $Y$. Suppose in addition we normalise so that
\[
  \frac{\sqrt{-1}}{2^3}\int_X \Omega \wedge \bar \Omega = 1, \qquad
  \int_Y \omega_Y = 1, \qquad  \frac{1}{2}\int_{X_y}\omega_X^2 = 1
\]
for all $y \in Y$. For $t>0$, let
$\tilde\omega_t \in [\omega_X+\frac{1}{t}\pi^*\omega_Y]$ be the unique
Ricci-flat K\"ahler metric such that
\[
  \frac{1}{3!}\tilde\omega_t^3 = \frac{\sqrt{-1}}{2^3}a_t\, \Omega\wedge\bar\Omega,
\]
where $a_t = O(t^{-1})$ is the positive constant determined by
integrating both sides of the equation. Note that by the uniqueness of
the Ricci-flat K\"ahler metrics, $\sigma$ is anti-symplectic:
\[
  \sigma^*\tilde\omega_t = -\tilde\omega_t.
\]

Now, let $L$ be a connected component of the fixed locus
$\operatorname{Fix}(\sigma) \subset X$. Then the following hold:
\begin{itemize}
\item $L$ is special Lagrangian with respect to $(X,\tilde\omega_t,\sqrt{a_t}\,\Omega)$.
\item For $y \in \gamma$, $L_y = \pi|_L^{-1}(y)$ is the real locus of
  $X_y$. Thus in particular $L_y$ is special Lagrangian with respect
  to the Calabi-Yau structure of the K3 surface $X_y$.
\item Each connected component of $L$ is a submersion onto $S^1$ whose
  fibres are compact Riemann surfaces.
\end{itemize}

Our goal is to prove the following, which is a reformulation of
Theorem~\ref{thm:main1}:

\begin{thm}\label{thm:oneforms}
  Let $X$ be a Calabi-Yau $3$-fold equipped with a Lefschetz K3
  fibration $\pi: X \to Y$. Under Assumptions~\ref{assumption:1}, for
  sufficiently small $t>0$, there exists a nowhere-vanishing harmonic
  $1$-form on the special Lagrangian submanifold $L$.
\end{thm}

To this end, we first recall the result of \cite{Li}, together with
the improved estimate in \cite{ChiuLin}.

Fix a local holomorphic coordinate $y$, and write
$\Omega = dy\wedge\Omega_y$. Define the generalised K\"ahler-Einstein
metric
\[
  \tilde\omega_Y = \frac{\sqrt{-1}}{2}A_y\, dy \wedge d\bar y,
\]
where
\[
  A_y = \frac{1}{2^2}\int_{X_y} \Omega_y \wedge \bar \Omega_y.
\]
Then $\operatorname{Ric}(\tilde\omega_Y) = \omega_{WP}$, the
Weil-Peterson metric of the fibration $\pi: X \to Y$.

For $y \in Y$, let
$\omega_y = \omega_X|_{X_y} + \sqrt{-1}\partial\bar\partial \psi_y$ be
the unique Ricci-flat K\"ahler metric on $X_y$ satisfying
\[
  \frac{1}{2!}\omega_y^2 = \frac{1}{2^2}\,A_y^{-1}\, \Omega_y \wedge\bar \Omega_y,
\]
so $(X_y,\omega_y,A_y^{-1/2}\Omega_y)$ is a Calabi-Yau $2$-fold. We
can then consider the \emph{semi-Ricci flat metric}
\[
  \omega_{SRF} = \omega_X + \frac{1}{t}\pi^*\tilde\omega_Y + \sqrt{-1}\partial\bar\partial\psi.
\]
This is a genuine K\"ahler metric away from the singular fibres of
$\pi: X\to Y$. In such a region, we have

\begin{prop}\label{prop:cylindrical}
  Let $y_0 \in Y \setminus S$, and let $B_{y_0}$ be a neighbourhood of
  $y_0$ with uniformly bounded diameter with respect to
  $\frac{1}{t}\tilde\omega_Y$. Then for \(t>0\) sufficiently small,
  there exists a diffeomorphism
  \[
    \Psi_{y_0}:B_{y_0}\times X_{y_0}\to \pi^{-1}(B_{y_0})
  \]
  such that
  \[
    \|\tilde\omega_t - \omega_{SRF}\|_{C^{k,\alpha}(\pi^{-1}(B_{y_0}), \tilde\omega_t)} = O(t^{\frac{1}{2}}), \\
    \|\Psi_{y_0}^*\Omega- dy\wedge\Omega_{y_0}\|_{C^{k,\alpha}(B_{y_0}\times X_{y_0})} = O(t^{\frac{1}{2}}), \\
    \|\Psi_{y_0}^*\omega_{SRF} - (t^{-1}\tilde\omega_Y+\omega_{y_0})\|_{C^{k,\alpha}(B_{y_0} \times X_{y_0})}
    = O(t^{\frac{1}{2}}),
  \]
  where in a local coordinate on $B_{y_0}$ we write
  $\Omega = dy \wedge \Omega_y$, the metric on
  $B_{y_0} \times X_{y_0}$ is the product metric
  $\frac{1}{t}\tilde\omega_Y+\omega_{y_0}$, and all H\"older norms are
  taken with respect to the indicated background metrics.
\end{prop}

\begin{remark}
  The last two estimates were proved by Li \cite{Li}, while the first
  estimate is the improved estimate in \cite[Proposition~3.9]{ChiuLin}.
\end{remark}

Our strategy for proving Theorem~\ref{thm:oneforms} is to first write
down a nowhere-vanishing closed $1$-form $\alpha$ that only depends on
the base variable along $\gamma$. $\alpha$ is almost harmonic in the
sense that $d^*\alpha = O(t^{1/2})$ and, more importantly, the
fibrewise average of $d^*\alpha$ is even smaller at $O(t)$. This
improved fibrewise-average estimate allows us to apply the parametrix
method and correct $\alpha$ to a nowhere-vanishing harmonic $1$-form
for all sufficiently small $t>0$.

\subsection{The ansatz nowhere-vanishing closed $1$-form}

Let $s\in [0,\ell]$ denote the arc-length parameter on $\gamma$ with
respect to the metric induced by $\frac{1}{t}\tilde \omega_Y$. Then
$ds$ is a nowhere vanishing $1$-form on $\gamma$. By abuse of
notation, we also write $ds$ for its pullback on $L$.

For $s\in [0,\ell]$, let $L_s = \pi|_L^{-1}(\gamma(s))$ denote the
corresponding fibre. $L_s$ is equipped with the metric $h_s$ induced
by the restriction of the Calabi-Yau metric $\tilde\omega_t$ on
$X$. We will also consider the metric $\tilde h_s$ induced by the
restriction of the semi-Ricci flat metric $\omega_{SRF}$. Unless
otherwise stated, all integrals over $L_s$ are taken with respect to
the volume form of the metric $h_s$.

We shall construct a $1$-form
\[
  \alpha = \varphi(s)\,ds
\]
which is almost harmonic in the following precise sense:

\begin{prop}\label{prop:error}
  For $t>0$ sufficiently small, there exists a nowhere-vanishing $1$-form
  $\alpha=\varphi(s)\,ds$ such that
  \[
    \|d^*\alpha\|_{C^{k,\alpha}(L)} = O(t^{\frac{1}{2}}),
  \]
  and moreover
  \[
    \left\|s\mapsto\int_{L_s} d^*\alpha\right\|_{C^{k,\alpha}(L)}
    = O\big(t^{\frac{1}{2}}\|d^*\alpha\|_{C^{k,\alpha}(L)}\big)
    = O(t).
  \]
  Here we regard the function $s\mapsto \int_{L_s} d^*\alpha$ as a
  function on $L$.
\end{prop}

\begin{remark}
  By contrast, if one simply takes the ansatz $1$-form to be $ds$,
  then although $d^*ds = O(t^{1/2})$, its fibrewise average is also
  $O(t^{1/2})$. See the proof of Proposition~\ref{prop:error} below.
\end{remark}

We need the following lemma:

\begin{lemma}\label{lemma:horizontal}
  There exists a lift of $\partial_s$ in $TL$, which we
  also denote by $\partial_s$, such that the following holds. Let
  $\Phi: [0,\ell] \times L_0 \to L$ be the flow along $\partial_s$ and
  let $x \in L_0$. The induced metric $g_L$ on $L$ by the restriction
  of the Ricci-flat K\"ahler metric $\tilde\omega_t$ can be written as
  \[
    \Phi^*g_L = a(x,s)\,ds^2 + ds\odot\beta_s(x) + h_s(x).
  \]
  Here $h_s = \Phi^*(\tilde g_t|L_s)$ and $\beta_s(x) \in \Lambda^1(L_0)$. The components admit
  the following expansions:
  \[
    a(x,s) &= 1 + O(t^{\frac{1}{2}}), \\
    \beta_s(x) &= O(t^{\frac{1}{2}}), \\
    h_s - \tilde h_s &= O(t^{\frac{1}{2}}),
  \]
  where the $O(t^{1/2})$ terms are measured in the $C^{k,\alpha}(L)$ for any fixed $k$ and $\alpha$.
\end{lemma}

To proceed, we first recall the integrability condition in
\cite{ChiuLin}. The general setting is the following.  Let $X$ be a
Calabi-Yau 3-fold equipped with a Lefschetz K3 fibration
$\pi: X \to Y$. Let $\gamma: [0,\ell] \to Y$ be a smooth path
parameterised by the arc length $s$ with respect to the metric induced
by $(1/t)\tilde\omega_Y$. Suppose for each $s \in [0,\ell]$ there is a
special Lagrangian submanifold $L_s$ in the corresponding K3 fibre
$X_s$, varying smoothly in $s$.

\begin{definition}
  \label{def:integrable}
  Let $X_\gamma = \cup_{s\in[0,\ell]} X_s$. Let
    $\tilde\gamma_1' \in \Gamma(TX_\gamma)$ be the lift of $\gamma'$
    that is orthogonal to $TX_s$ for $s\in[0,\ell]$ with respect to
    $g_{SRF}$. We say $L$ satisfies the integrability condition if
    there exists a lift $\tilde\gamma_2' \in \Gamma(TL)$ of $\gamma'$
    such that the following hold. For every $s \in [0,\ell]$,
    $\omega_s(\tilde\gamma_2'(s)-\tilde\gamma_1'(s), \cdot)|_{L_s}$ is
    $L^2$-orthogonal to the space of harmonic $1$-forms on
    $L_s \subset X_s$ with respect to $\tilde h_s$, the metric induced
    by the semi-Ricci flat metric $g_{SRF}$.
\end{definition}

\begin{remark}
  Although $\omega_{SRF}$ depends on $t$, the $1$-dimensional
  horizontal distributions $\mathbb{R}\langle \tilde\gamma_1'\rangle$
  and $\mathbb{R}\langle \tilde\gamma_2'\rangle$ do not depend on $t$.
\end{remark}

\begin{lemma}\label{lemma:integrable}
  If $L$ is integrable in the sense of
  Definition~\ref{def:integrable}, then $g_L$ satisfies the properties
  in Lemma~\ref{lemma:horizontal}.
\end{lemma}

\begin{proof}
  We choose the lift $\partial_s$ from Lemma~\ref{lemma:horizontal} to be $\partial_s=\tilde{\gamma}_2' \in \Gamma(TL)$.
  The estimates for $h_s-\tilde h_s$ follow directly from
  Proposition~\ref{prop:cylindrical}. We focus on the estimates for
  $a(x,s)$ and $\beta_s(x)=\beta(s,x)$.
  
  Locally near $s_0 \in [0,\ell]$, let $F_s: L_{s_0} \to X_s$ with
  $F_s(L_{s_0}) = L_s$ be the coordinate map such that
  $\tilde\gamma_2' = \partial_s F_s$. Let
  $V_s = \tilde\gamma_2' - \tilde\gamma_1'$ denote the projection onto
  the vertical component $TX_s$.

  Define $\alpha_s = \iota_{V_s}\omega_s|_{L_s}$. Since $L_s$ is
  Lagrangian with respect to $\omega_s$, we have
  \[
    F_s^*\omega_s = 0.
  \]
  Differentiating in $s$, we get
  \[
    0 = \left.\frac{d}{ds}\right|_{s=s_0} F_s^*\omega_s = d\alpha_{s_0} + \left.\nabla_s \omega_s\right|_{s=s_0},
  \]
  where $\nabla_s$ is the fibrewise derivative using the horizontal
  lift. It follows that
  \[
    d_{s_0}\alpha_{s_0} = - \left.\nabla_s \omega_s\right|_{s=s_0}.
  \]
  Here $d_s$ denotes the differential on $L_s$. By
  Proposition~\ref{prop:cylindrical}, we have
  \[
    \|d_{s_0}\alpha_{s_0}\|_{C^{k,\alpha}(L_{s_0})} \le Ct^{\frac{1}{2}}.
  \]

  Next, since $L_s$ is a special Lagrangian submanifold in
  $(X_s,\omega_s,A_{s_0}^{-1/2}\Omega_{s_0})$, we have
  \[
    F^*_s\operatorname{Im}(e^{-i\phi(s)}\Omega_s) = 0
  \]
  for some $\phi:\gamma \rightarrow \R$.
  Differentiating in $s$, we get
  \[
    0 = \left.\frac{d}{ds}\right|_{s=s_0}F_s^*\operatorname{Im}(e^{-i\phi(s)}\Omega_s)
    = d\iota_{V_{s_0}}\operatorname{Im}(e^{-i\phi(s_0)}\Omega_{s_0})
    + \left.\nabla_s\operatorname{Im}(e^{-i\phi(s)}\Omega_s)\right|_{s=s_0}.
  \]
  It follows that
  \[
    d_{s_0}^*\alpha_{s_0} = -*\left.\nabla_s\operatorname{Im}(e^{-i\phi(s)}\Omega_s)\right|_{s=s_0}.
  \]
  Note that $\phi'(s)$ has size of the same order as $\nabla_s \Omega_s$. By
  Proposition~\ref{prop:cylindrical}, we have
  \[
    \|d_{s_0}^*\alpha_{s_0}\|_{C^{k,\alpha}(L_{s_0})} \le Ct^{\frac{1}{2}}.
  \]
  Here $d_s^*$ is with respect to $\tilde h_s$.

  Now, by the assumption that $\alpha_s$ is $L^2$-orthogonal to the
  kernel of $D_s = d_s+d_s^*$, the coercive estimate gives
  \[
    \|\alpha_s\|_{C^{k+1}(L_s)} \le C\|D_s\alpha_s\|_{C^{k,\alpha}(L_s)}
    \le Ct^{\frac{1}{2}},
  \]
  where the constant $C>0$ depends on $L$ but is independent of $t$.

  Next, we estimate the derivatives in $s$. For a $1$-form $\eta_s(x)$
  on $L_s$, let $\Pi_s\eta_s$ denote the $L^2$ projection onto the
  space of harmonic $1$-forms on $(L_s,\tilde h_s)$. We have the
  following orthogonal decomposition:
  \[
    \partial_s\alpha_s = (1-\Pi_s)\partial_s\alpha_s + \Pi_s(\partial_s\alpha_s).
  \]

  The first term on the right hand side of the decomposition can be
  estimated as follows. First we note that
  \[
    \partial_s(D_s\alpha_s) = (\partial_sD_s)\alpha_s + D_s(\partial_s\alpha_s),
  \]
  and therefore
  \[
    D_s(\partial_s\alpha_s) = \partial_s(D_s\alpha_s)-(\partial_sD_s)\alpha_s.
  \]
  It follows that
  \[
    \|(1-\Pi_s)\partial_s\alpha_s\|_{C^{k,\alpha}(L_s)}
    &\le C\|D_s(1-\Pi_s)\partial_s\alpha_s\|_{C^{k-1,\alpha}(L_s)} \\
    &\le C\left(\|\partial_s(D_s\alpha_s)\|_{C^{k-1,\alpha}(L_s)} 
    +\|(\partial_sD_s)\alpha_s\|_{C^{k-1,\alpha}(L_s)} \right)\\
    &\le Ct. \\
  \]

  For the second term, note that
  \[
    0 = \partial_s(\Pi_s\alpha_s) = (\partial_s\Pi_s)\alpha_s + \Pi_s(\partial_s\alpha_s),
  \]
  and therefore
  \[
    \Pi_s(\partial_s\alpha_s) = -(\partial_s\Pi_s)\alpha_s.
  \]
  It follows that
  \[
    \|\Pi_s(\partial_s\alpha_s)\|_{C^{k,\alpha}(L_s)} \le \|(\partial_s\Pi_s)\alpha_s\|_{C^{k,\alpha}(L_s)}
    \le Ct.
  \]

  Combining the above estimates, we conclude that
  \[
    \|\partial_s\alpha_s\|_{C^{k,\alpha}(L_s)} \le Ct.
  \]
  For higher derivatives in $s$, we can proceed by induction.

  Having estimated $\alpha_s$, it is immediate from the definition that
  $V_s = O(t^{1/2})$ in $C^{k,\alpha}(L)$ by taking the dual of
  $\alpha_s$ with respect to $\omega_{SRF}$.

  To estimate $a(x,s)$, we use \ref{prop:cylindrical} and compute
  \[
    a(x,s) &= g_L(\tilde\gamma_2', \tilde\gamma_2') \\
    &= \tilde g_t(\tilde\gamma_1' + V_s, \tilde\gamma_1' + V_s) \\
    &= g_{SRF}(\tilde\gamma_1' + V_s, \tilde\gamma_1' + V_s) + O(t^{1/2}) \\
    &= 1 + O(t^{1/2}).
  \]

  Similarly, to estimate $\beta_s(x)$, fix $W_s \in TL_s$ with
  $|W_s| = O(1)$ and compute
  \[
    \beta_s(W) &= g_L(\tilde\gamma_2', W_{\color{red}s}) \\
    &= \tilde g_t(\tilde\gamma_1'+ V_s, W_{\color{red}s}) \\
    &= g_{SRF}(\tilde\gamma_1'+ V_s, W_{\color{red}s}) + O(t^{1/2}) \\
    &= O(t^{1/2}).
  \]

  This completes the proof.
\end{proof}

\begin{proof}[Proof of Lemma~\ref{lemma:horizontal}]
  By Lemma~\ref{lemma:integrable}, it is enough to show that $L$
  satisfies the integrability condition in the sense of
  Definition~\ref{def:integrable}.
  
  Let $\tilde\gamma_1'$ be the unique vector field such that
  $\tilde\gamma_1'(s) \in (TX_s)^\perp$ with respect to the semi-Ricci
  flat metric $\omega_{SRF}$. For $s \in [0,\ell]$, we
    write $\sigma_s = \sigma|_{X_s}$. Note that by the uniqueness of the
  Calabi-Yau metric, we have
  \[
    \sigma_s^*\omega_s = -\omega_s,
  \]
  and thus the fibrewise Riemannian metric is invariant under pullback
  by $\sigma_s$.
  By equivariance,
  \[
    \pi_*\sigma_{*}\tilde\gamma_1' = \sigma_{*}\pi_*\tilde\gamma_1'
    =\sigma_{*}\gamma'
    = \gamma'.
  \]
  It follows from uniqueness of the orthogonal decomposition that
  \[
    \sigma_{s*}\tilde\gamma_1' = \tilde\gamma_1',
  \]
  and thus
  \[
    \tilde\gamma_1' \in TL.
  \]

  Next, we simply let
  \[
    \tilde\gamma_2' = \tilde\gamma_1'|_L.
  \]
  Then we immediately have
  \[
    \omega_s(\tilde\gamma_2'-\tilde\gamma_1',\cdot) = 0,
  \]
  which completes the proof.
\end{proof}

\begin{proof}[Proof of Proposition~\ref{prop:error}]
  Choose a local coordinate system $x^i$ on $L_0$. By
  Lemma~\ref{lemma:horizontal}, we can write
  \[
    g_L = a\,ds^2 + 2\beta_i\,ds\,dx^i + h_{ij}\,dx^i\,dx^j.
  \]
  
  Let $\beta^\# =\beta^i\partial_{x^i}= h^{ij}\beta_j\partial_{x^i}$
  be the dual of $\beta$ with respect to $h$. Then we can diagonalise
  \[
    a\,ds^2 + 2\beta_i\,ds\,dx^i + h_{ij}\,dx^i\,dx^j =
    \lambda (e^0)^2 + h_{ij}e^i\,e^j,
  \]
  where $\lambda = a - |\beta|_h^2$ and the coframe is given by
  $e^0 = ds, e^i = dx^i+\beta^ids$. The dual frame is then given by
  $e_0 = \partial_s - \beta^i\partial_{x^i},
  e_i=\partial_{x^i}$. Using the dual frame, the dual metric is given
  by
  \[
    g^{-1} = \lambda^{-1} (e_0)^2 + h^{ij}e_i\,e_j,
  \]

  Now, for a $1$-form $\alpha = \alpha_\mu e^\mu$, we have
  \[
    d^*\alpha = -\frac{1}{\sqrt{\det g}}e_\mu\left(\sqrt{\det g} \alpha^\mu\right),
  \]
  where $\det g = \lambda \det h$.
  For $\alpha = f(s)\,ds$, we have $\alpha^\# = \lambda^{-1}f(s)e_0$.

  Fix $s_0 \in [0,\ell]$. Thanks to Lemma~\ref{lemma:horizontal} and
  Proposition~\ref{prop:cylindrical}, we have
  \[
    \|h_s - h_{s_0}\|_{C^{k,\alpha}(L)} \le Ct^{\frac{1}{2}}
  \]
  for $s \in (s_0-\varepsilon,s_0+\varepsilon)$ for some fixed
  $\varepsilon>0$. Here we use the horizontal lifting to identify
  $\pi|_{L}^{-1}(s_0-\varepsilon,s_0+\varepsilon)$ with
  $(s_0-\varepsilon,s_0+\varepsilon) \times L_{s_0}$, and the formula
  for $d^*\alpha$ simplifies to
  \[\label{eq:laplacian}
    d^*\alpha = -\frac{1}{\sqrt{\lambda}\sqrt{\det h}}
    (\partial_s-\beta^\#)\left(\sqrt{\lambda^{-1}}\sqrt{\det h} \,\varphi\right).
  \]
  By the divergence formula, integrating over a fibre $L_s$ yields
  \[
    \int_{L_s} \lambda^{\frac{1}{2}}d^*\alpha =
    -\partial_s\int_{L_s} \lambda^{-\frac{1}{2}}\varphi.
  \]

  Setting
  \[
    \varphi(s) = \left(\int_{L_s}\lambda^{-\frac{1}{2}}\right)^{-1},
  \]
  we see that
  \[
    \int_{L_s} \lambda^{\frac{1}{2}}d^*\alpha = 0.
  \]
  
  By Lemma~\ref{lemma:horizontal}, it is easy to see that for
  sufficiently small $t > 0$, $\lambda = 1 + O(t^{1/2})$ and
  $\varphi(s)$ is nowhere vanishing. Combining these, we conclude that
  \[
    \int_{L_s} d^*\alpha = O(t^{\frac{1}{2}}\|d^*\alpha\|_{C^{k,\alpha}(L)}),
  \]
  together with higher order estimates.
  
  Finally, by Lemma~\ref{lemma:horizontal} and
  Proposition~\ref{prop:cylindrical}, we have $\beta = O(t^{1/2})$ and
  $\partial_s h_s = O(t^{1/2})$. Thus, by \eqref{eq:laplacian}, we
  conclude that $\|d^*\alpha\|_{C^{k,\alpha}(L)} = O(t^{1/2})$.
\end{proof}

We write
\[
  f = d^*\alpha
\]
for the coclosedness error of $\alpha$. By Hodge theory, there exists
$\hat\eta \in C^{k+1,\alpha}(L)$ such that $d\hat\eta = 0$ and
$d^* \hat\eta = f$. Then the $1$-form
\[
  \eta = \alpha - \hat\eta
\]
is harmonic. Indeed,
\[
  (d+d^*)\hat\eta = d^*\alpha - d^*\hat\eta = f - d^*\hat\eta = 0.
\]

The difficulty is that the coercive estimate for the Hodge-Dirac
operator $d+d^*$ is not uniform in $t$. Since the diameter of $L$ is
$O(t^{-1/2})$, one has
\[
  \|\hat\eta\|_{C^{k+1,\alpha}(L)} \le Ct^{-\frac{1}{2}}\|(d+d^*) \hat\eta\|_{C^{k,\alpha}(L)} = O(1),
\]
so a priori it is not clear if $\eta$ is a small perturbation of
$\alpha$. In the following, we solve the equation with
  estimates using a parametrix construction. Namely, we solve the
  equation on model cylinders (see Lemma~\ref{lemma:cylinder}) and on
  the ``base direction''. Gluing these local solutions together, we
  obtain a parametrix. The parametrix can then be corrected by a
  Neumann series to obtain an actual right inverse with the desired
  estimates. The smallness of the fibrewise average established in
  Proposition~\ref{prop:error} is crucial for showing that the
  solution is small.

\subsection{The parametrix}

We now state our main technical result:

\begin{prop}
  \label{prop:parametrix}
  For $t>0$ sufficiently small, the following holds. Let
  $f \in C^{k,\alpha}(L)$ such that
  \[
    \int_L f = 0, \qquad
    \left\|s\mapsto\int_{L_s} f\right\|_{C^{k,\alpha}(L)} =
    O(t^{\frac{1}{2}}\|f\|_{C^{k,\alpha}(L)}).
  \]
  Then there exist
  $u_{fib} \in C^{k+2,\alpha}(L)$ and $u_{base} \in C^{k+2,\alpha}(L)$ which is
  constant along fibres such that $u= u_{fib} + u_{base}$ satisfies
  $\Delta u = f$ and that
  \[
    \|u_{fib}\|_{C^{k+2,\alpha}(L)} + \|du_{base}\|_{C^{k+1,\alpha}(L)}  \le C\|f\|_{C^{k,\alpha}(L)}
  \]
\end{prop}

\begin{proof}[Proof of Theorem~\ref{thm:oneforms}]
  By Propositions \ref{prop:error} and \ref{prop:parametrix}, there exists
  \[
    u = u_{fib} + u_{base} \in C^{k+2,\alpha}(L),
  \]
  with $u_{base}$ depending only on the base variable $s$, such that
  \[
    \Delta u = d^*\alpha
  \]
  together with the estimate
  \[
    \|u_{fib}\|_{C^{k+2,\alpha}(L)} + \|du_{base}\|_{C^{k+1,\alpha}(L)} \le Ct^{\frac{1}{2}}.
  \]
  It follows that
  \[
    \|du\|_{C^{k+1,\alpha}(L)} \le Ct^{\frac{1}{2}}.
  \]
  Thus, for sufficiently small $t>0$, $\eta = \alpha - du$ remains
  nowhere vanishing, since $\alpha = O(1)$ is nowhere vanishing. This
  completes the proof.
\end{proof}

The rest of this section is dedicated to the proof of
Proposition~\ref{prop:parametrix}. Suppose $f \in C^{k,\alpha}(L)$
satisfies the assumptions in Proposition~\ref{prop:parametrix}. Define
\[
  f_{[0]} = f.
\]
The parametrix we construct has two components: the local inverses
applied to the localised, fibrewise average-zero part of the source
term, and the inverse applied to the fibrewise average part, which
depends only on the base variable $s$.

We will be using the notions from the proof of
Proposition~\ref{prop:error}. For $s \in [0,\ell]$, define
\[
  W(s) = \int_{L_s} \lambda^{\frac{1}{2}}, \qquad \tilde V(s) = \int_{L_s} \lambda^{-\frac{1}{2}}.
\]
The weighted fibrewise average of $f_{[0]}$ is defined to be
\[
  \bar f_{[0]} = \frac{1}{W(s)}\int_{L_s} \lambda^{\frac{1}{2}}f_{[0]}.
\]

By assumption, we have
\[
  \int_{\gamma} W(s)\bar f_{[0]} \,ds =
  \int_L f_{[0]} = 0, \qquad
  \|\bar f_{[0]}\|_{C^{k,\alpha}(\gamma)} \le Ct^{\frac{1}{2}}\|f_{[0]}\|_{C^{k,\alpha}(\gamma)}.
\]

To invert $\bar f_{[0]}$, we solve the ODE
\[
  \Delta_{base} u(s) = -\frac{1}{W(s)}\frac{d}{ds}\Big(\tilde V(s)\frac{d}{ds} u(s)\Big) = \bar f_{[0]}(s).
\]

\begin{remark}\label{remark:laplacian}
  The linear operator $\Delta_{base}: C^\infty(\gamma) \to C^\infty(\gamma)$ defined as
  \[
      \Delta_{base}u(s) = -\frac{1}{W(s)}\frac{d}{ds}\Big(\tilde V(s)\frac{d}{ds} u(s)\Big)
  \]
  is the fibrewise averaged Laplacian with respect to the weight
  $\lambda^{1/2}$. To see this, let $u(s)$ be a function on
  $\gamma$. Then $du = u'(s)\,ds$. From \eqref{eq:laplacian}, we see
  that
  \[
    \Delta u = -\frac{1}{\sqrt{\lambda}\sqrt{\det h}}
    (\partial_s-\beta^\#)\left(\lambda^{-\frac{1}{2}}\sqrt{\det h} u'(s)\right).
  \]

  Taking the weighted average over a fibre yields
  \[
    \frac{1}{W(s)}\int_{L_s} \lambda^{\frac{1}{2}}\Delta u
    =
    -\frac{1}{W(s)}\frac{d}{d s}\int_{L_s} \lambda^{-\frac{1}{2}}u'(s)
    =
    -\frac{1}{W(s)}\frac{d}{d s} \left(\tilde V(s) \frac{d}{ds}u(s)\right).
  \]
  So  $\Delta_{base}$ is precisely the fibrewise-averaged $\Delta$.
\end{remark}

To solve the ODE, note that since $\int_\gamma W(s)\bar f_{[0]} = 0$, there
is a unique solution
\[
  u_{[0],base}(s) \in C^{k+2,\alpha}(\gamma)
\]
with
$u_{[0],base}(0) = 0$ such that
\[
  \|du_{[0],base}\|_{C^{k+1,\alpha}(\gamma)} \le C\ell\|\bar f_{[0]}\|_{C^{k,\alpha}(\gamma)}.
\]
Since $u_{[0],base}$ is constant along fibres and
$\ell = O(t^{-1/2})$, we conclude that
\[\label{eq:du}
  \|du_{[0],base}\|_{C^{k+1,\alpha}(L)} \le Ct^{-\frac{1}{2}}\|\bar f_{[0]}\|_{C^{k,\alpha}(L)}
\]

We have
\[
  \|\Delta u_{[0],base} - \bar f_{[0]}\|_{C^{k,\alpha}(L)}
  &\le \|(\Delta - \Delta_{base}) u_{[0],base}\|_{C^{k,\alpha}(L)} \\
  &\le C\|\bar f_{[0]}\|_{C^{k,\alpha}(L)} \\
  &\le Ct^{\frac{1}{2}}\|f_{[0]}\|_{C^{k,\alpha}(L)}.
\]
Furthermore, by Remark~\ref{remark:laplacian}, we have
\[
  \int_{L_s} \lambda^{\frac{1}{2}}(\Delta u_{[0],base} - \bar f_{[0]}) = 0.
\]
It follows that
\[
  \int_{L_s} (\Delta u_{[0],base} - \bar f)  
  &=\int_{L_s} (1-\lambda^{\frac{1}{2}})(\Delta u_{[0],base} - \bar f) \\
  &= O\Big(t^{\frac{1}{2}}\|\bar f_{[0]}\|_{C^{k,\alpha}(L)}\Big).
\]

Now, let us define
\[
  \tilde f_{[0]} = f_{[0]} - \Delta u_{[0],base},
\]
so
\[
  \int_{L_s} \lambda^{\frac{1}{2}} \tilde f_{[0]}
  =
  \int_{L_s} \lambda^{\frac{1}{2}} \big[(f_{[0]} - \bar f_{[0]}) + (\bar f_{[0]} - \Delta u_{[0],base})\big] = 0
\]
for all $s \in [0,\ell]$. We also have the following general estimate:

\begin{lemma}
  \[
    \|\Delta u_{[0],base}\|_{C^{k,\alpha}(L)}\le C\|\bar f_{[0]}\|_{C^{k,\alpha}(L)}.
  \]
\end{lemma}

\begin{proof}
  Since we have the estimate \eqref{eq:du}, this estimate boils down
  to estimating the coefficients of $\Delta - \Delta_{base}$ acting on
  functions that are constant along the fibres.
\end{proof}

Let $\chi_i$ be a partition of unity of $\gamma$ subordinate to
neighbourhoods $U_i$ with size $\sim 1$ and centres $s_i \in
[0,\ell]$. We also view $\chi_i$ as a partition of unity on
$L$. Define
\[
  f_i = \chi_i \tilde f_{[0]}.
\]
Using the local trivialisation induced from the horizontal lifting,
$f_i$ is identified with a function on the cylinder
$\mathbb{R} \times L_{s_i}$.

\begin{lemma}\label{lemma:small_fiberwise}
\[
  \big\|\int_{L_{s_i}} f_i(s, \cdot)\big\|_{C^{k,\alpha}(\mathbb{R}\times L_{s_i})} \le Ct^{\frac{1}{2}}\|f_i\|_{C^{k,\alpha}(L)}.
\]
\end{lemma}

\begin{proof}
  Let $\Phi_{s,s_i}: L_{s_i} \to L_{s}$ be the local trivialisation
  induced by the horizontal lifting. For $s \in \gamma$ in the support
  of $\chi_i$, we have
  \[
    \|\Phi_{s,s_i}^*f_i - f_i(s_i)\|_{C^{k,\alpha}(L)} \le Ct^{\frac{1}{2}}\|f_i\|_{C^{k,\alpha}(L)}.
  \]
  Also, since $\int_{L_s} \lambda^{\frac{1}{2}}f_i(s) = 0$, it follows
  that
  \[
    \big|\int_{L_{s_i}} f_i(s,\cdot)\big| =
    \big|\int_{L_{s_i}} (\Phi_{s,s_i}^*f_i - f_i(s_i) + (1-\lambda^{\frac{1}{2}})f_i(s_i))\big|
    \le Ct^{\frac{1}{2}}\|f_i\|_{C^{k,\alpha}(L)},
  \]
  and higher order estimates follow similarly.
\end{proof}

Let
\[
  \tilde f_i = f_i - \frac{1}{V(s)}\int_{L_s} f_i,
\]
where $V(s)$ Then we have $\int_{L_s} \tilde f_i = 0$. We apply the
following general result of spectral decomposition, which can for
example be found in \cite{Li}.

\begin{lemma}\label{lemma:cylinder}
  Let \(M\) be a compact Riemannian manifold, and let
  \(X=\mathbb{R}\times M\) be equipped with the product metric
  \(ds^2+g_M\). Define
  \[
    C^{k,\alpha,ave}(X)
    =
    \left\{
      f\in C^{k,\alpha}(X)
      \;\middle|\;
      \int_M f(s,\cdot)=0 \text{ for all } s\in\mathbb{R}
    \right\}.
  \]
  Then the Laplacian
  \[
    \Delta:C^{k+2,\alpha,ave}(X)\to C^{k,\alpha,ave}(X)
  \]
  is invertible with bounded inverse.

  Moreover, let \(0<m<\sqrt{\lambda_1(M)}\), where \(\lambda_1(M)\) is
  the first positive eigenvalue of \(\Delta_M\). If
  \(f\in C^{k,\alpha,ave}(X)\) is supported in \(\{|s|<B\}\)
  for some \(B>0\), and \(u=\Delta^{-1}f\), then
  \[
    \|e^{m(|s|-B)}u\|_{C^{k+2,\alpha}(X)}
    \le C(M,k,\alpha,m)\|f\|_{C^{k,\alpha}(X)}.
  \]
\end{lemma}

\begin{proof}
  For the reader's convenience, we include a proof here.

  Let $\{\phi_j\}_{j\ge 1}$ be an $L^2$-orthonormal basis of
  eigenfunctions of $\Delta_M$ on $C^{\infty,ave}(M)$, with
  \[
    \Delta_M \phi_j = \lambda_j \phi_j,
    \qquad
    0<\lambda_1\le \lambda_2\le \cdots .
  \]
  Since $f\in C^{k,\alpha,ave}(X)$, we may expand
  \[
    f(s,x)=\sum_{j=1}^\infty f_j(s)\phi_j(x).
  \]
  We seek a solution of the form
  \[
    u(s,x)=\sum_{j=1}^\infty u_j(s)\phi_j(x).
  \]
  Then \(\Delta u=f\) is equivalent to the family of ODEs
  \[
    -u_j''(s)+\lambda_j u_j(s)=f_j(s).
  \]

  For each $j\ge 1$, define
  \[
    u_j(s) =
    \frac{1}{2\sqrt{\lambda_j}}
    \int_{\mathbb{R}} e^{-\sqrt{\lambda_j}|s-t|} f_j(t)\,dt .
  \]
  A direct computation shows that $u_j$ solves
  \[
    -u_j''+\lambda_j u_j=f_j.
  \]
  Since $\lambda_j\ge \lambda_1>0$, the kernel
  \[
    K_j(s,t)=\frac{1}{2\sqrt{\lambda_j}}e^{-\sqrt{\lambda_j}|s-t|}
  \]
  is uniformly integrable in $j$, so the operator
  $f_j\mapsto u_j$ is bounded on H\"older spaces, uniformly in $j$.
  Summing over $j$ gives a bounded inverse
  \[
    \Delta^{-1}: C^{k,\alpha,ave}(X)\to C^{k+2,\alpha,ave}(X).
  \]

  Now assume that $f$ is supported in $\{|s|<B\}$. Fix
  $0<m<\sqrt{\lambda_1}$. Since $\lambda_j\ge \lambda_1$, for
  $|t|\le B$ we have
  \[
    e^{m(|s|-B)}e^{-\sqrt{\lambda_j}|s-t|}
    \le
    e^{m(|s|-B)}e^{-\sqrt{\lambda_1}(|s|-|t|)}
    \le
    e^{-(\sqrt{\lambda_1}-m)(|s|-B)} \le 1 .
  \]
  Therefore
  \[
    e^{m(|s|-B)}|u_j(s)|
    \le
    \frac{1}{2\sqrt{\lambda_j}}
    \int_{|t|<B} |f_j(t)|\,dt
    \le
    C(\lambda_1,m,B)\|f_j\|_{C^0(\mathbb{R})}.
  \]
  The same estimate applies to $u_j'$ and $u_j''$, since
  differentiating the kernel only introduces factors of
  $\sqrt{\lambda_j}$. Summing over $j$ and applying standard Schauder
  estimates on the product cylinder yields
  \[
    \|e^{m(|s|-B)}u\|_{C^{k+2,\alpha}(X)}
    \le
    C(M,k,\alpha,m)\|f\|_{C^{k,\alpha}(X)}.
  \]
  This proves the weighted estimate.
\end{proof}

By Lemma~\ref{lemma:cylinder}, there exists a unique
$u_i \in C^{k+2,\alpha}(\mathbb{R} \times L_{s_i})$ with
$\int_{L_{s_i}} u_i = 0$ such that
$\Delta_{\mathbb{R}\times L_{s_i}} u_i = \tilde f_i$ with
$\|u_i\|_{C^{k+2,\alpha}} \le C\|f_i\|_{C^{k,\alpha}}$, where $C>0$
depends on $L$.

To transplant these $u_i$ back to $L$, we consider a set of cutoff
functions $\tilde\chi_i$ on $\gamma$ such that $\tilde\chi_i = 1$ on
the support of $\chi_i$ and the support of $\tilde\chi_i$ is of size
$\Lambda \sim |\log t| \ll t^{-1/2}$.

Define \[
  u_{[0],fib} = \sum_i \tilde\chi_i u_i.
\]

\begin{lemma}\label{lemma:fiberwise}
  For $s \in [0,\ell]$, we have
  \[
    \|\Delta u_{[0],fib} - \tilde f_{[0]}\|_{C^{k,\alpha}(L)}
    \le Ct^{\frac{1}{2}}\|\tilde f_{[0]}\|_{C^{k,\alpha}(L)}.
  \]
\end{lemma}

\begin{proof}
  There are three contributions to the residual: (1) the metric deviation and (2) the cutoff
  error. Since
  $\|u_i\|_{C^{k+2,\alpha}(L)} \le C\|\tilde f\|_{C^{k,\alpha}(L)}$
  and the metric deviation is $O(t^{1/2})$, we have
  \[
    \|(\Delta - \Delta_{\mathbb{R}\times L_{s_i}})u_i\|_{C^{k,\alpha}} \le
    Ct^{\frac{1}{2}}\|\tilde f_i\|_{C^{k,\alpha}(L)}.
  \]
  We now turn to the cutoff error. The key is to notice that on the
  support of $d\tilde\chi_i$, we have
  \[
    \|u_i\|_{C^{k+2,\alpha}(L)} \le Ce^{-m\Lambda}\|\tilde f_i\|_{C^{k,\alpha}(L)}.
  \]
  If we choose $\Lambda \sim |\log t|$ so that
  $\exp(-m\Lambda) \le t^{1/2}$, then we satisfy the desired estimate.
\end{proof}

We can now write down our parametrix:
\[
  u_{[0]}  =  u_{[0],base} + u_{[0],fib}.
\]

Define
\[
  f_{[1]} = f_{[0]} - \Delta u_{[0]}
\]
to be the residual. Then $f_{[1]}$ satisfies
\[
  \|f_{[1]}\|_{C^{k,\alpha}(L)} &\le
  \|\bar f_{[0]} - \Delta u_{[0],base}\|_{C^{k,\alpha}(L)}
  + \|\tilde f_{[0]}-\Delta u_{[0],fib}\|_{C^{k,\alpha}(L)} \\
  &\le Ct^{\frac{1}{2}}\|f_{[0]}\|_{C^{k,\alpha}(L)}
\]
by Lemma~\ref{lemma:fiberwise}. It should be noted that unlike
$f_{[0]}$, $f_{[1]}$ does not necessarily have small fibrewise average
compared to the size of itself. However, we can still iterate the
construction of the parametrix and obtain
\[
  u_{[n]} = u_{[n],base} + u_{[n],fib}
\]
and
\[
  f_{[n+1]} = f_{[n]} - \Delta u_{[n]}.  
\]
with
\[
  \|du_{[n],base}\| &\le Ct^{-\frac{1}{2}}\|f_{[n]}\|_{C^{k,\alpha}(L)}, \\
  \|u_{[n],fib}\| &\le C\|f_{[n]}\|_{C^{k,\alpha}(L)}, \\
  \|f_{[n+1]}\|_{C^{k,\alpha}(L)} &\le Ct^{\frac{1}{2}}\|f_{[n]}\|_{C^{k,\alpha}(L)}
\]
for $n \ge 1$. The iteration is contractive once $t>0$ is sufficiently
small.

Finally, define
\[
  u_{base} &= \sum_{n=0}^\infty u_{[n],base}, \\
  u_{fib} &= \sum_{n=0}^\infty u_{[n],fib}, \\
  u &= u_{base} + u_{fib}.
\]
Then $u$ satisfies the required properties. The crucial point is that,
since the original error $f$ has fibrewise average of size
\[
  O(t^{1/2}\|f\|_{C^{k,\alpha}(L)}),
\]
we have
\[
  u_{[0],base},\,u_{[1],base}=O(\|f\|_{C^{k,\alpha}(L)})
\]
and
\[
  u_{[n],base}&=O\bigl(t^{\frac{n-1}{2}}\|f\|_{C^{k,\alpha}(L)}\bigr)
\]
for $n \ge 2$. This completes the proof of
Proposition~\ref{prop:parametrix}.

\section{Generic smoothness of certain fibrations}
\label{sec:k3}
It is a well-known consequence of Bertini's theorem that a general hyperplane section of  smooth projective variety $X$ is again a smooth projective variety.  It is not, however, the case that if $X$ is equipped with a smooth fibration $X\to S$, then a general hyperplane section of $X$ is equipped with a smooth fibration to $S$.  Nevertheless, one expects that for a general hyperplane section $H$ of $X$, the fibres of $H \to S$ have mild singularities.  The purpose of this section is to record a result which is certainly well-known to experts, namely, that $H$ can be chosen so that the fibres of $H \to S$ are either smooth or have ordinary double points (ODPs).  In our current setting we are also interested in ensuring that our hyperplane section respects some additional structure (namely the real structure on our varieties) and this presents some additional minor complications.

We say that a variety $X$ defined over $\mathbb C$ has at worst a single ODP if either $X$ is smooth, or ${\rm Sing}(X) = \{x\}$ is a single point and that point is an ODP, i.e., local analytically isomorphic to a 
singularity of the form $\{x_1^2+\dots+x_n^2 = 0\} \subset \mathbb C^n$.

\begin{prop}
\label{prop_lefschetz}
Let $p\colon X \to Y$ be a smooth morphism with connected fibres between complex projective manifolds.
Let $\eta$ be a line bundle on $X$ and suppose that the following hold:
\begin{enumerate}
    \item[(a)] for all $y \in Y$, there exists a very ample line bundle $M_y$ on $X_y =p^{-1}(y)$ and a positive integer $d \ge 2$ such that $\eta|_{X_y} \sim M_y^{\otimes d}$; and

    \item[(b)] 
    $p_*\eta$ is globally generated and $H^i(X_y, \eta|_{X_y}) = 0$ for all $i>0$ and $y \in Y$.
\end{enumerate}

Then, there exists a Zariski open subset $U \subset H^0(X, \eta)$ such that for any $s \in U$ the following hold:
\begin{enumerate}
    \item $W=\{s = 0\}$ is smooth; and 
    \item there exists a Zariski open subset $V \subset Y$ such that ${\rm codim}_Y(Y\setminus V) \ge 2$ and for all $y \in V$ the fibre of $p|_{W}\colon W \to Y$ over $y$ has at worst a single ODP, in particular is a Lefschetz fibration (cf. Def. \ref{def:lefschetz-fibration}). 
\end{enumerate}

Suppose moreover that the following hold:
\begin{enumerate}
    \item[(c)] $X, Y$ and $\eta$ can be defined over $\mathbb R$
    and the projection $p\colon X \to Y$ is defined over $\mathbb R$ (in particular, the morphism $p\colon X \to Y$ commutes with complex conjugation); 
    \item[(d)] there exists $s_0 \in H^0(X, \eta)(\mathbb R)$ such that if $W_0=\{s_0 = 0\}$, then the discriminant locus of $p|_{W_0}\colon W_0 \to Y$ is disjoint from $Y(\mathbb R)$; and

    \item[(e)] 
    $W_0$ is smooth and $W_0(\mathbb R) \neq \emptyset$.
\end{enumerate}

Then, we may choose $s \in U$ as above such that the following additional properties hold:

\begin{itemize}
    \item[(3)] $W$ is defined over $\mathbb R$ and $p|_{W}\colon W \to Y$ is defined over $\mathbb R$ (in particular, it commutes with complex conjugation);
    \item[(4)] $W(\mathbb R)$ is non-empty; and 
    \item[(5)] the discriminant locus of $W \to Y$ is disjoint from $Y(\mathbb R)$.
\end{itemize}
    
\end{prop}
\begin{proof}
We first prove the existence of a Zariski open subset $U$ as claimed.  To prove the existence of $U$ we may freely replace $Y$ by a general complete intersection curve $\Sigma \subset Y$ and so may assume that $\dim Y =1$.

Consider the subset
\begin{align*} 
U_0= \{(s, y) \in H^0(X, \eta) \times Y : \{s = 0\} \cap X_y ~\text{has at worst a single ODP} \} \\\subset H^0(X, \eta) \times Y.\end{align*}
We claim that $U_0$ is a Zariski open subset.  Recall that for any flat proper morphism $f\colon V \to T$ the set \[\{t \in T: V_t ~\text{has at worst a single ODP}\} \subset T\] is Zariski open (this follows, for instance, by examining the versal deformation space of the ODP, cf. 
\cite[Corollaire 1.3.4]{Deligne73} and also \cite[Proposition 3.2]{Katz1973}).  Applying this observation to the 
 universal family $\pi\colon \mathscr V \to  Y \times  H^0(X,  \eta)$ (i.e., the fibre of $\pi$ over $(y, s)$ is precisely
$X_y \cap \{s = 0\}$) shows that $U_0$ is open.
Let $Z = H^0(X, \eta) \times Y \setminus U_0$ and let $p_1\colon Z \to  H^0(X, \eta)$, $p_2\colon Z \to Y$ be the two projections.

For any $y \in Y$, since $\eta|_{X_y}$ is at least twice a very ample line bundle, we may apply \cite[Corollaire 3.2.4]{Katz1973}
to conclude that the set 
\[\{s \in H^0(X_y, \eta|_{X_y}) : \{s = 0\} ~\text{has at worst a single ODP} \} \subset H^0(X_y, \eta|_{X_y})\]
is a Zariski open subset and that its complement, which we denote $Z_y$, has codimension at least $2$.

We next observe that the restriction map 
    $\rho_y\colon H^0(X, \eta) \to H^0(X_y, \eta|_{X_y})$ is surjective for all $y \in Y$.
Indeed, since $p$ is flat $\chi(X_y, \eta_{X_y})$ is constant as a function of $y$ and since $H^i(X_y, \eta|_{X_y}) = 0$ by assumption it follows that $h^0(X_y, \eta|_{X_y})$ is
also a constant function of $y$. 
We may apply \cite[III Corollary 12.9]{Hartshorne1977} to see that $p_*\eta$ is a vector bundle and the fibre of $p_*\eta$ over $y$
is precisely 
$H^0(X_y, \eta|_{X_y})$.
Observe that $\rho_y$ is therefore precisely the composition
\[H^0(X, \eta) = H^0(Y, p_*\eta) \to (p_*\eta)_y= H^0(X_y, \eta|_{X_y})\]
and the middle arrow is surjective because $p_*\eta$ is globally generated, by assumption.

Since $\rho_y\colon H^0(Y, \eta) \to H^0(X_y, \eta|_{X_y})$ is surjective it follows that 
$\rho_y^{-1}(Z_y)$ is codimension at least two in $H^0(X, \eta)$.
Observe that by definition $p_2^{-1}(y) \subset \rho_y^{-1}(Z_y)$ and thus
$p_2^{-1}(y)$ is codimension at least two in $H^0(X, \eta)$.
It follows that \[\dim Z \leq \dim Y + H^0(X, \eta) - 2 = H^0(X, \eta) -1.\] Since $p_1$ is a proper morphism, $p_1(Z) \subset H^0(X, \eta)$ is a closed subset and by construction, for any $s \in H^0(X, \eta) \setminus p_1(Z)$, $\{s=0\}$ satisfies (2).  

By hypothesis (b) and the assumption that $\eta$ is $p$-very ample  we have that $\eta$ is base point free. By Bertini's theorem, perhaps replacing $H^0(X, \eta) \setminus p_1(Z)$ by a smaller Zariski open subset $U \subset H^0(X, \eta) \setminus p_1(Z)$ we may assume that (1) holds for all $s \in U$.  We have therefore found our desired open subset.

\medskip

We now show that under the additional hypotheses listed in the statement of the Proposition we can find $s \in U$ satisfying properties (3), (4) and (5).

Since the real points of $H^0(X, \eta)$
are Zariski dense, we see that set of real points has non-empty intersection with $U$ and so we may find $s_1 \in U$ such that $s_1$ is defined over $\mathbb R$.

Let $\epsilon$ be a real parameter. We will show for a sufficiently small choice of of $\epsilon$ that $s = s_0+\epsilon s_1$ satisfies conditions (1)-(5). 

For all but finitely many values of $\epsilon$, $s_0+\epsilon s_1 \in U$, and so if we let $W = \{s_0+\epsilon s_1 = 0\}$ then $W$ satisfies conditions (1) and (2).

Note also that $W$ is defined over $\mathbb R$, and hence the morphism $p|_W\colon W \to Y$ is defined over $\mathbb R$
and therefore commutes with complex conjugation.  This shows that (3) holds.


Since $W_0$ is smooth, it follows that $W_0(\mathbb R)$ is a real manifold with (real) dimension equal to the (complex) dimension of $W_0$, see \cite[Proposition 2.2.27]{MR4179588}.  For $\epsilon$ sufficiently small it follows that $W(\mathbb R)$ is also a real manifold with $\dim W_0(\mathbb R) = \dim W(\mathbb R)$, in particular, it is non-empty.

Finally, since the discriminant locus of $W_0 \to Y$ is disjoint from $Y(\mathbb R)$ and the discriminant locus of $\{s_0+\epsilon s_1 = 0\} \to Y$ varies continuously with $\epsilon$, we see that
by taking $\epsilon$ to be sufficiently small we can ensure that (5) holds as well.
\end{proof}

\begin{remark}
\label{remark:alternative-generic-odp-conditions}
Recall that if $X$ is a smooth variety defined over $\mathbb R$ and $X(\mathbb R) \neq \emptyset$, then $X(\mathbb R)$ is Zariski dense in $X$, cf. \cite[Theorem 2.2.9]{MR4179588}.  It follows that hypothesis (e) can also be replaced by either of the following alternate conditions.

\begin{enumerate}
    \item[(e')] For some $y \in Y(\mathbb R)$, $W_0 \cap X_y$ is smooth, $\dim W_0 \cap X_y \ge 1$
    and $W_0(\mathbb R) \cap X_y \neq \emptyset$.
\end{enumerate}

Or

\begin{enumerate}
    \item[(e'')] For some $y \in Y$, $\# (X_y \cap W_0(\mathbb R)) = +\infty$.
\end{enumerate}

  \medskip

\end{remark}

\begin{lemma}
    \label{lefschetz_lemma_branched_cover}
    Let $p\colon X \to Y$ be a smooth morphism with connected fibres between complex projective manifolds, suppose that $X, Y$ and $p$ can be defined over $\mathbb R$, let $\eta_0$ be a line bundle on $X$ which is defined over $\mathbb R$, let $\eta = \eta_0^{\otimes 2}$ and let $0 \neq s \in H^0(X, \eta)$ be a section such that $W = \{s = 0\}$ satisfies conditions (1)-(5) in Proposition \ref{prop_lefschetz}.

    Let $\sigma\colon \tilde X \to X$ be the double cover extracting a root of $s$ and which is totally ramified over $W$, see \cite[Definition 2.50]{KM98}, and let $\tilde p = p\circ \sigma$.
    Then $\tilde p\colon \tilde X \to Y$ satisfies conditions (1)-(5) in 
Proposition \ref{prop_lefschetz}.
\end{lemma}
\begin{proof}
As in the proof of Proposition \ref{prop_lefschetz} we may freely reduce to the case where $\dim Y = 1$.

\medskip

  The fact that conditions (3) and (4) hold for $\tilde p\colon \tilde X \to Y$  is immediate from the construction given in \cite[Definition 2.50]{KM98}.
  
\medskip
   
   We now verify conditions (1) and (2) hold for $\tilde p\colon \tilde X \to Y$. This can be done locally on a small neighbourhood about  a point  $x \in U \subset X$ in which case we may freely assume that we have (holomorphic) coordinates $(z_1, \dots, z_n)$ on $U$ such that $p(z_1, \dots, z_n) = (z_2, \dots, z_n)$,  $W = \{f(z_1, \dots, z_n) = 0\}$
   and $\tilde U = \sigma^{-1}(U) \cong \{(w, z_1, \dots, z_n) : w^2 -f(z_1, \dots, z_n) = 0\}$.
From this local description it is immediate that $\tilde U$ is smooth, and that the fibre of $\tilde U \to Y$ has an ODP at $\sigma^{-1}(x)$  if and only if the fibre of $W \to Y$ has an ODP at $x$.

\medskip

The above local description also shows that the discriminant locus of $\tilde X \to Y$ is precisely the discriminant locus of $W \to Y$, and hence condition (5) holds.
\end{proof}

\section{Examples and applications to $G_2$ geometry}
\label{sec:examples}

In this section we construct examples for Theorem \ref{thm:main1}.
They are Calabi-Yau $3$-folds which are K3-fibred over $\P^1$ and have the property that their discriminant locus has no real points.
If some additional technical conditions are satisfied, then Proposition \ref{prop_lefschetz} can be applied to perturb the variety so that all fibre singularities are ODPs and Theorem \ref{thm:main1} can be applied.
The result is a Calabi-Yau $3$-fold with a nowhere vanishing harmonic $1$-form on its real locus.
The construction from \cite{JoyceKarigiannis} can then be applied, see Corollary \ref{cor:joyce-karigiannis-for-cy}, and we obtain a $G_2$-manifold and can compute its Betti numbers.

We find that several different constructions of Calabi-Yau manifolds can be used to produce the inputs for our construction and exhibit the following examples:
hypersurfaces in a product of projective spaces;
a complete intersection Calabi-Yau in a product of projective spaces;
a hypersurface in a scroll;
a branched double cover of $\P^1 \times \P^2$.

The Betti numbers listed in Corollary \ref{cor:main2} are precisely the ones obtained in the examples below.

\begin{ex}
    \label{example:1}
    On $\P^1 \times \P^3$ with coordinates $[s:t],[x_0:x_1:x_2:x_3]$ let 
    \begin{align}
    \label{equation:example1}
    f=(s^2+t^2) x_0^4
    -
    (s^2+2t^2) x_1^4
    +
    (s^2+3t^2) x_2^4
    -
    (s^2+5t^2) x_3^4,
    \quad
    X = Z(f).
    \end{align}
    \emph{Real structure and fibration:}
    Let $\sigma: \P^1 \times \P^3 \rightarrow \P^1 \times \P^3$ denote complex conjugation.
    Complex conjugation on $\P^1$ and $\sigma$ are intertwined by the projection.
    \emph{Line bundle:}
    Let $\eta=\mathcal{O}(2,4)$.
    Then $\eta|_{\P^3}=\mathcal{O}(4)=\mathcal{O}(1)^{\otimes 4}$, so (a) from Proposition \ref{prop_lefschetz} is satisfied.
    By K\"unneth, $H^0(\P^1 \times \P^3, \mathcal{O}(2,4)) \cong H^0(\P^1, \mathcal{O}(2)) \otimes H^0(\P^3, \mathcal{O}(4))$, and restriction to a fibre $\{[s:t]\} \times \P^3$ corresponds to evaluation at $[s:t]$ on the first component tensored with the identity on the second component, so is surjective.
    Hence, (b) from Proposition \ref{prop_lefschetz} is also satisfied.
    \emph{Non-empty real locus:}
    We have $([1:0],[1:1:0:0]) \in X$, so its real locus is non-empty.    
    \emph{Real discriminant locus:}
    Let $s,t \in \R$, then the coefficients of $x_0^4, x_1^4, x_2^4, x_3^4$ are non-zero, hence the fibre of $Z(f)$ over $[s:t]$ is from the well-known Fermat family which is smooth.
    \emph{Smooth total space:}
    denote the $s,t$-polynomial factors in front of $x_i$ by $q_i(s,t)$.
    At a singular point, we have for $i \in \{0,1,2,3\}$ that either $x_i=0$ or $q_i(s,t)=0$.
    But the four equations $q_i(s,t)=0$ are pairwise incompatible in $\P^1$, so at most one $x_i$ can be nonzero.
    Because $[x_0:x_1:x_2:x_3] \in \P^3$, exactly one must be nonzero.
    That would give a singular point of $q_i$, but it is a smooth quadric.

    Thus, the conditions (a)-(e) of Proposition \ref{prop_lefschetz} are satisfied, and a small perturbation of $f$ has the property that all singular fibres have exactly one ODP.
    Thus, we can apply Theorem \ref{thm:main1} and find a metric $\tilde{g}_t$ such that $L=\fix(\sigma)$ admits a nowhere vanishing $1$-form.
    Thus, Corollary \ref{cor:joyce-karigiannis-for-cy} applies and we obtain a $G_2$-manifold $N$ by resolving the singularities of $(S^1 \times X)/\langle (-1) \times \sigma \rangle$.

    \emph{Computation of $G_2$-topology:}
    By \cite{Green1987}, $b^2(X)=2$.
    Two generators of $H^2(X \times S^1)$ are the pullbacks of the Kähler forms from $\P^1$ and $\P^3$.
    On both, $\sigma$ acts as multiplication by $-1$, hence $\dim (H^{1,1}(X)^\sigma)=0$ and $\dim (H^{1,1}(X)^{-\sigma})=2$.
    We also have $h^{2,1}=86$ by \cite{Green1987}.

    We now compute the Betti numbers of $L=\fix(\sigma)$.
    The map 
    \begin{align*}
        \Psi: \RP^1 \times \RP^3 &\rightarrow \RP^1 \times \RP^3
        \\
        ([s:t],[x_0:x_1:x_2:x_3]) & \mapsto 
        ([s:t],[q_0^{1/4} x_0:q_1^{1/4} x_1:q_2^{1/4} x_2:q_3^{1/4} x_3]),
    \end{align*}
    where $q_i$ are the quadratic factors defined above, is well-defined and a diffeomorphism (for its inverse replace $q_i^{1/4}$ with $q_i^{-1/4}$).
    We have $\Psi(L)=\RP^1 \times \Sigma$ for $\Sigma=\{[y_0:y_1:y_2:y_3] \in \RP^3: y_0^4-y_1^4+y_2^4-y_3^4=0\}$.
    We prove $\Sigma \simeq T^2$ by considering $C=\{(u,v) \in \R^2:u^4+v^4=1\} \simeq S^1$.
    Then $\Phi: C \times C \rightarrow \Sigma$, $\Phi((u_0,u_2),(u_1,u_3))=[u_0:u_1:u_2:u_3]$ is surjective with fibres consisting of two points obtained by simultaneously multiplying all inputs by $-1$.
    Hence, $\Sigma \simeq (C \times C)/\{\pm 1\} \simeq (S^1 \times S^1)/\{\pm 1\}=T^2$.
    Altogether, $L \simeq T^3$, so $b^0(L)=2$, $b^1(L)=6$ and therefore for the resolved $G_2$-manifold $N$ we find
    $b^2(N)=0+2=2$, $b^3(N)=89+6=95$.
    A $G_2$-manifold with these Betti numbers is listed in \cite[Table 1]{KovalevLee}.
\end{ex}

\begin{ex}
    \label{example:cicy2}
    This example shows that one single CICY configuration can lead to different $G_2$-manifolds.
    The application of Theorem \ref{thm:main1} and Proposition \ref{prop_lefschetz} and Corollary \ref{cor:joyce-karigiannis-for-cy} are the same, so it remains to compute the Betti numbers of the resulting $G_2$-manifold $N$.
    We change \eqref{equation:example1} from Example \ref{example:1} by changing one minus sign into a plus sign:
    \begin{align}
    f=(s^2+t^2) x_0^4
    -
    (s^2+2t^2) x_1^4
    +
    (s^2+3t^2) x_2^4
    -
    (s^2+5t^2) x_3^4,
    \quad
    X = Z(f).
    \end{align}
    The whole computation remains the same, except that now $\Sigma \simeq S^2$, in the notation from Example \ref{example:1}.
    Hence, $b^2(N)=2$ and $b^3(N)=89+2=91$.
    As before, a $G_2$-manifold with these Betti numbers is listed in \cite[Table 1]{KovalevLee}.
\end{ex}

\begin{ex}
    \label{example:kaihnsa-10-s2-example}
    We now use the same CICY configuration to construct an example with potentially new topology.
    It also shows how real algebraic geometry results can be used for this construction method.
    As in the previous example, the application of Theorem \ref{thm:main1} and Proposition \ref{prop_lefschetz} and Corollary \ref{cor:joyce-karigiannis-for-cy} are the same as in Example \ref{example:1}, so it remains to compute the Betti numbers of the resulting $G_2$-manifold $N$.
    Let $[s:t],[x_0:x_1:x_2:x_3]$ be coordinates on $\P^1 \times \P^3$ as before.
    Let $\sigma_i$ be the $i$-th elementary symmetric polynomial in $x_0, x_1, x_2, x_3$ and set
    \begin{align*}
        h &=
        6 \sigma_1^4-53 \sigma_1^2 \sigma_2+120 \sigma_2^2 - 320 \sigma_4,
        \\
        k &=
        x_0^4+x_1^4+x_2^4+x_3^4,
        \\
        f &=
        (s^2+t^2) h+\epsilon t^2 k,
        \quad
        X=Z(f),
    \end{align*}
    for $0<\epsilon \ll 1$ to be determined later.
    Here $h$ is taken from \cite[Example 6.2]{Kaihnsa2019} which itself refers to \cite[Section 9]{Rohn1913}.

    \emph{Non-empty real locus:}
    $Z_{\R}(h)$ is non-empty by \cite[Example 6.2]{Kaihnsa2019}, say $x\in Z_{\R}(h)$.
    Then $([1:0],x) \in X$.
    \emph{Real discriminant locus:}
    for $[s:t] \in \RP^1$ the fibre of $Z((s^2+t^2)h)$ over $[s:t]$ is isomorphic to $Z(h)$, and a routine calculation shows that it is smooth.
    The property that fibres are smooth is an open condition, so for small $\epsilon$ the fibres of $X$ over real points remain smooth.
    \emph{Smooth total space:}
    Another routine check shows that $Z(h,k)\subset \P^3$ is a smooth
    complete intersection. 
    Since smoothness of hypersurfaces is an open condition in the space of quartic forms, $Z(h+\epsilon k)$ is smooth for all sufficiently small $\epsilon$.
    Suppose $X=Z((s^2+t^2)h+\epsilon t^2k)$ is singular at some $([s:t],x)$.
    The equations
    $f_s=0$ and $f_t=0$ give $2sh=0$ and $2t(h+\epsilon k)=0$.
    If $t=0$, then $h=dh=0$, contradicting smoothness of $Z(h)$. If $s=0$, then
    $h+\epsilon k=d(h+\epsilon k)=0$, contradicting smoothness of $Z(h+\epsilon k)$.
    Finally, if $s,t\neq 0$, then $h=k=0$. 
    If $s^2+t^2\neq 0$, then $dh$ and $dk$ are linearly dependent, contradicting
    smoothness of $Z(h,k)$. If $s^2+t^2=0$, the same equation gives $dk=0$,
    which is impossible since $Z(k)$ is smooth. Hence $X$ is smooth.
    
    \emph{Computation of $G_2$-topology:}
    The computation is identical to Example \ref{example:1} except for the computation of the real locus.
    By \cite[Example 6.2]{Kaihnsa2019}, the real quartic 
    $Z_\R(h)\subset \RP^3$
    is the disjoint union of ten $2$-spheres. 
    The polynomial $s^2+t^2$ has no real zeros, so $Z_{\R}((s^2+t^2)h)={\RP}^1 \times Z_{\R}(h)$ and is smooth because of smoothness of $h$.
    As before, the left hand side is diffeomorphic to $Z_{\R}(f)$ for $\epsilon$ small enough, so altogether $L \simeq \RP^1 \times V_\R(h) \simeq \bigsqcup_{i=1}^{10} (S^1\times S^2)$.
    Hence, $b^0(L)=10$, $b^1(L)=10$.
    So, by Corollary \ref{cor:joyce-karigiannis-for-cy},
    $b^2(N)=0+20=20$
    and $b^3(N)=89+20=109$.
    This pair of Betti numbers does not appear in \cite[Table 1]{KovalevLee}, \cite[Theorems 4.2, 4.6, 4.11]{reidegeld}, \cite[Table 2]{JoyceII} so it may be new.
\end{ex}

\begin{ex}
    Here we give an example using another CICY configuration.
    Let $[s:t],[x_0:\dots:x_4]$ be coordinates on $\P^1 \times \P^4$.
    Let $R \in \R[x_1,x_2,x_3,x_4]$ be a smooth real quartic such that $Z_{\R}(R) \cong S^2_{\# 2} \cup 6S^2$, where $S^2_{\# 2}$ denotes a $2$-sphere with two handles added.
    Such a quartic exists by \cite[Section 3.4]{Kharlamov1976}.
    Define
    \begin{align*}
        l=(s^2+t^2)x_0+(s^2-t^2)x_1+stx_2, \quad
        B=Z(l) \subset \P^1 \times \P^4, \quad
        X_0 = Z(l,R) \subset \P^1 \times \P^4.
    \end{align*}
    \emph{Real structure and fibration:}
    as in the previous examples.
    \emph{Line bundle:}
    we will apply Proposition \ref{prop_lefschetz} to $p: B \rightarrow \P^1$.
    To this end, let $\eta = \mathcal{O}_B(0,4)$.
    The coefficients $s^2+t^2$ and $s^2-t^2$ and $st$ do not vanish simultaneously, so $B_y \simeq \P^3$ for all $y \in \P^1$.
    Thus, $\eta|_{B_y} \cong \mathcal{O}_{\P^3}(4)=\mathcal{O}_{\P^3}(1)^{\otimes 4}$, so (a) holds.
    By the definition of ideal sheaf and \cite[Proposition II.6.18]{Hartshorne1977} we have the short exact sequence for the divisor $B_y$:
    \[
    0
    \rightarrow \mathcal{O}_{\P^4}(3)
    \rightarrow \mathcal{O}_{\P^4}(4)
    \rightarrow \mathcal{O}_{B_y}(4)
    \rightarrow 0.
    \]
    Because $H^1(\P^4, \mathcal{O}_{\P^4}(3))=0$, the long exact sequence in cohomology gives that $H^0(\P^4,\mathcal{O}(4)) \rightarrow H^0(B_y, \eta|_{B_y})$ is surjective.
    This map factors as $H^0(\P^4,\mathcal{O}(4)) \rightarrow H^0(B,\eta) \rightarrow H^0(B_y, \eta|_{B_y})$ and because the whole map is surjective, the second map must be surjective, thereby satisfying (b).

    \emph{Real discriminant locus:}
    For $[s:t] \in \RP^1$ we have $s^2+t^2>0$, so we can solve $l=0$ for $x_0$ and projection onto the last four coordinates gives an isomorphism $B_{[s:t]} \rightarrow \P^3_{[x_1:x_2:x_3:x_4]}$.
    Under this isomorphism, the fibre of $X_0$ over $[s:t]$ is the smooth quartic $Z(R)$, hence the discriminant locus of $X_0 \rightarrow \P^1$ has no real points.
    Also, $X_0(\R) \cong \RP^1 \times Z_{\R}(R) \cong S^1 \times (S^2_{\# 2} \cup 6S^2)$, so the real locus is non-empty.
    Thus, the conditions (a)-(e) of Proposition \ref{prop_lefschetz}are satisfied, and a small perturbation of $X_0$, say $X$, has the property that all singular fibres have exactly one ODP.
    As in the previous examples, we can apply Theorem \ref{thm:main1} and find a metric $\tilde{g}_t$ such that $L=\fix(\sigma)$ admits a nowhere vanishing $1$-form and Corollary \ref{cor:joyce-karigiannis-for-cy} to get a $G_2$-manifold $N$.

    \emph{Computation of $G_2$-topology:}
    by \cite{Green1987}, $h^{1,1}(X)=2$ and $h^{2,1}(X)=86$.
    As before, $\dim H^{1,1}(X)^\sigma=0$ and $H^{1,1}(X)^{-\sigma}=2$.
    From the computation of the real locus above, we have for $L=\fix(\sigma) \subset X$ that $b^0(L)=7$ and $b^1(L)=11$.
    Thus, by Corollary \ref{cor:joyce-karigiannis-for-cy}:
    $b^2(N)=14$ and $b^3(N)=111$.
    As before, this pair of Betti numbers may be new.
\end{ex}

\begin{ex}
    This example uses a special Lagrangian that is a nontrivial fibration over $S^1$.
    Consider the unweighted scroll $\pi: \mathbb{F}=\mathbb{P}_{\mathbb{P}^1}(\mathcal{O}(1) \oplus \mathcal{O} \oplus \mathcal{O}(1) \oplus \mathcal{O}) \rightarrow \P^1$.
    Let $[s:t]$ be coordinates on the base and $x_0,x_1,x_2,x_3$ be relative homogeneous coordinates on $\mathcal{O}(1) \oplus \mathcal{O} \oplus \mathcal{O}(1) \oplus \mathcal{O}$.
    Define
    \begin{align*}
    q_0 &= (s^2+t^2)(s^2+2t^2),\quad
    q_2 = (2s^2+t^2)(3s^2+t^2),
    \\
    Y &= Z(q_0 x_0^4+x_1^4-q_2 x_2^4-x_3^4)
    \subset \mathbb{F}.
    \end{align*}
    Then $Y$ is Calabi-Yau by \cite[Theorem 5.1.1]{Mboya2023}.

    \emph{Real structure and fibration:}
    projection onto the base $\pi: \mathbb{F} \rightarrow \mathbb{P}^1$ defines the fibration and complex conjugation $\sigma:\mathbb{F} \rightarrow \mathbb{F}$ restricts to $Y$, which commutes with complex conjugation on $\P^1$ under $\pi$.
    \emph{Line bundle:}
    Set $\eta =\mathcal{O}_{\mathbb{F}}(4)=-K_{\mathbb{F}}$, where $\mathcal{O}_{\mathbb{F}}(1)$ is the tautological bundle from \cite[Section 3.4]{Mboya2023} and \cite[Section 27.21]{StacksProject}.
    Then, for $y \in \P^1$, $\eta|_{\mathbb{F}_y} = \mathcal{O}_{\mathbb{F}}(4)|_{\mathbb{F}_y} \simeq \mathcal{O}_{\P^3}(4)$, so (a) of Proposition \ref{prop_lefschetz} is satisfied.
    To check (b):
    writing $E=\mathcal{O}(1) \oplus \mathcal{O} \oplus \mathcal{O}(1) \oplus \mathcal{O}$, we have $\pi_*\mathcal{O}_{\mathbb{F}}(4) \cong \operatorname{Sym}^4E$ by \cite[Exercise III.8.4(a)]{Hartshorne1977} or \cite[Example 27.21.3]{StacksProject} and the bundle $E$, and therefore $\operatorname{Sym}^4 E$, is globally generated.
    For $y \in \P^1$ we have $\eta|_{\mathbb{F}_y} \cong \mathcal{O}_{\P^3}(4)$, hence $H^i(\mathbb{F}|_y, \eta|_{\mathbb{F}_y})=0$ for $i > 0$ by the standard cohomology of projective space.

    \emph{Non-empty real locus:}
    over the affine open $U_0 = \{[s:t]:s \neq 0\}$ the bundle $\mathcal{O}(1)$ is trivialised by the section $s$ and $[x_0:x_1:x_2:x_3]$ become homogeneous coordinates of the fibre $Y_{[s:t]}$.
    In these coordinates, $([s:t],[x_0:x_1:x_2:x_3])=([1:0],[1:0:0:1])$ can be seen to be a real point on $Y_0$ which can be perturbed to a point on $Y$.
    \emph{Real discriminant locus:}
    over the affine open $U_0$, the equation for the fibre over $[s:t]$ is a quartic K3 surface which is seen to have a singularities if $q_0=0$ or $q_2=0$, which has no real solutions.
    Analogously, one checks the second affine open.
    \emph{Smooth total space:}
    At $[0:1] \in \P^1$ we have that $q_0$ and $q_2$ are non-zero, so $Y_0$ is smooth.
    It remains to check smoothness of $Y_0$ over $U_0$.
    For $z=t/s$ the defining equation becomes
    \[
    f=(1+z^2)(1+2z^2)x_0^4+x_1^4
      -(2+z^2)(3+z^2)x_2^4-x_3^4.
    \]
    If $q_0 \neq 0$ and $q_2 \neq 0$, then the fibre derivatives being zero gives $x_0=x_1=x_2=x_3=0$, which is a contradiction.
    Thus, $q_0=0$ or $q_2=0$.
    Assume $df(z,[x_0:x_1:x_2:x_3])=0$.
    The zeros of $q_0$ and $q_2$ are simple and disjoint.
    If $q_0=0$, then $q_2 \neq 0$ and the fibre derivatives force $[x_0:x_1:x_2:x_3]=[1:0:0:0]$.
    But then $df=0$ gives $q_0'(z)=0$, contradicting simplicity of zeros of $q_0$.
    The remaining case is checked analogously.

    Thus, the conditions (a)-(e) of Proposition \ref{prop_lefschetz} are satisfied, and as before, Corollary \ref{cor:joyce-karigiannis-for-cy} applies, and we obtain a $G_2$-manifold $N$ by resolving the singularities of $(S^1 \times X)/\langle (-1) \times \sigma \rangle$.

    \emph{Computation of $G_2$-topology:}
    by \cite[Remark 3.2]{Mboya2024}, the Picard number of $Y$ is $\rho(Y)=\dim_{\mathbb{Q}} (H^2(Y,\mathbb{Q}) \cap H^{1,1}(Y))=2$.
    As $Y$ is Calabi-Yau, we have $h^{2,0}(Y)=0$, so $h^{1,1}(Y)=2$.
    Two linearly independent elements are given by the restrictions of the ambient divisors $h=c_1(\pi^*\mathcal{O}_{\P^1}(1))$ and $\xi=c_1(\mathcal{O}_{\mathbb{F}}(1))$.
    If confusion is unlikely, we also write $h$ for $c_1(\mathcal{O}_{\P^1}(1))$.
    
    To check linear independence, consider $ah|_Y+b\xi|_Y=0$ for $a,b \in \C$.
    Since $\mathbb{F}_y \simeq \P^3$ we have $\int_{\mathbb{F}_y} \xi^3=1$, i.e. $\pi_*(\xi^3)=1 \in A^0(\P^1)$, where $A^i$ denote the Chow groups and $\pi_*$ is the integration-over-fibres map.
    Thus, by the projection formula \cite[Lemma 43.27.1(3)]{StacksProject}, we have $\int_{\mathbb{F}} h \xi^3 = \int_{\P^1} h \pi_*(\xi^3)=\int_{\P^1}h=1$.
    Pairing $ah|_Y+b\xi|_Y=0$ with $h\xi|_Y$ gives $0=\int_Y (ah+b\xi)h\xi=\int_{\mathbb{F}} (ah+b\xi)h\xi \cdot 4\xi=4b$, using $h^2=0$ and $[Y]=4\xi$ (see \cite[p.626]{Mboya2024}), hence $b=0$.
    Pairing with $\xi^2|_Y$ shows $a=0$, which proves linear independence.
    
    Complex conjugation acts as $-1$ on both, hence 
    \[
      \dim H^{1,1}(Y)^\sigma=0, \qquad \dim H^{1,1}(Y)^{-\sigma}=2.
    \]

    We now compute $h^{2,1}(Y)$.
    The Chern classes of a projectivised bundle satisfy \cite[Equation 42.37.1.1]{StacksProject}, which in our case gives
    \[
    \xi^4-\pi^* c_1(E) \xi^3+\pi^* c_2(E) \xi^2-\pi^* c_3(E) \xi+\pi^* c_4(E)=0,
    \]
    where
    \[
    E=\mathcal{O}(1) \oplus \mathcal{O} \oplus \mathcal{O}(1) \oplus \mathcal{O}.
    \]
    Since $E$ is a bundle on $\P^1$, we have $c_1(E)=2h$ and $c_i(E)=0$ for $i \geq 2$, hence the equation gives $\xi^4=2h \xi^3$.
    Hence, $\int_{\mathbb{F}} \xi^4=2 \int_{\mathbb{F}} h \xi^3=2$.
    Taking the dual of \cite[Exercise III.8.4.b]{Hartshorne1977} gives the relative Euler sequence for $\pi: \mathbb{F} \rightarrow \P^1$:
    \[
    0
    \rightarrow \mathcal{O}_{\mathbb{F}}
    \rightarrow \pi^* E^* \otimes \mathcal{O}_{\mathbb{F}}(1)
    \rightarrow T_{\mathbb{F}/\P^1}
    \rightarrow 0.
    \]
    Here, $E^*=\mathcal{O}(-1) \oplus \mathcal{O} \oplus \mathcal{O}(-1) \oplus \mathcal{O}$, hence $\pi^* E^* \otimes \mathcal{O}_{\mathbb{F}}(1) \cong (\mathcal{O}_{\mathbb{F}}(1) \otimes \pi^* \mathcal{O}_{\P^1}(-1))^{\oplus 2} \oplus \mathcal{O}_{\mathbb{F}}(1)^{\oplus 2}$.
    Thus $c(T_{\mathbb{F}/\P^1})=(1+\xi-h)^2(1+\xi)^2$.
    Applying \cite[Section 50.12]{StacksProject} to $\pi: \mathbb{F} \rightarrow \P^1$ and taking its dual gives the short exact sequence
    \[
    0
    \rightarrow T_{\mathbb{F}/\P^1}
    \rightarrow T_{\mathbb{F}}
    \rightarrow \pi^* T_{\P^1}
    \rightarrow 0.
    \]
    Using $c(T_{\P^1})=1+2h$ and the above equation for $c(T_{\mathbb{F}/\P^1})$, by taking total Chern classes in this short exact sequence we get $c(\mathbb{F})=(1+2h)(1+\xi-h)^2(1+\xi)^2$.
    Expanding and using $h^2=0$ we get $c_2(\mathbb{F})=6\xi^2+2h\xi$ and $c_3(\mathbb{F})=4\xi^3+6h\xi^2$.
    The normal sequence
    \[
    0
    \rightarrow T_Y
    \rightarrow T_{\mathbb{F}}|_Y
    \rightarrow \mathcal{O}_Y(Y)
    \rightarrow 0
    \]
    gives $c(T_Y)=\frac{c(T_{\mathbb{F}})}{1+4\xi}\big|_Y=(1+c_1(\mathbb{F})+c_2(\mathbb{F})+c_3(\mathbb{F})+\dots)(1-4\xi+16\xi^2-64\xi^3 \pm \dots)|_Y$, where the degree three term is $c_3(T_Y)=(c_3(\mathbb{F})-c_2(\mathbb{F}) 4 \xi+c_1(\mathbb{F}) 16 \xi^2-64 \xi^3)|_Y$.
    The relation $c_1(\mathbb{F})=-K_{\mathbb{F}}=[Y]=4\xi$ gives $c_1(\mathbb{F}) 16 \xi^2-64 \xi^3=0$ and plugging this in yields $c_3(Y)=(c_3(\mathbb{F})-c_2(\mathbb{F})4 \xi)|_Y$.
    Therefore, 
    \begin{align*}
        \chi(Y)
        &=
        \int_Y c_3(Y)=\int_{\mathbb{F}}(c_3(\mathbb{F})-4\xi c_2(\mathbb{F})) 4 \xi
        =
        \int_{\mathbb{F}}(4\xi^3+6h\xi^2-4\xi (6\xi^2+2h\xi)) 4 \xi
        \\
        &=
        \int_{\mathbb{F}}
        (-20 \xi^3-2h\xi^2)4\xi
        =
        -80 \cdot 2-8 \cdot 1 
        =
        -168.
    \end{align*}
    On the other hand, because $Y$ is Calabi-Yau, $\chi(Y)=2(h^{1,1}(Y)-h^{2,1}(Y))$, and $h^{1,1}(Y)=2$ gives $h^{2,1}(Y)=86$.
    
    We now compute the topology of $\fix (\sigma)=L$.
    The real bundles $\mathcal{O}_{\RP^1}(1)$ and $\mathcal{O}_{\RP^1}(-1)$ are isomorphic, so $\mathbb{F}_{\R} \cong \P_{\RP^1}(\mathcal{O}(-1) \oplus \mathcal{O} \oplus \mathcal{O}(-1) \oplus \mathcal{O})$.
    Define the circle $C=\{(u_0,u_1) \in \R^2: u_0^4+u_1^4=1\}$ and $T=(C \times C)/\{\pm 1\}_{\text{diag}}$.
    Define
    \begin{align*}
        \Phi: [0,\pi] \times T &\rightarrow L \subset \P_{\RP^1}(\mathcal{O}(-1) \oplus \mathcal{O} \oplus \mathcal{O}(-1) \oplus \mathcal{O})
    \end{align*}
    by
    \begin{align*}
        &(\theta,[u,v]) = \\
        &\left(
        [\cos \theta:\sin \theta],
        \left[
        (q_0(\cos \theta, \sin \theta))^{-1/4} u_0 e_\theta:
        u_1:
        (q_2(\cos \theta, \sin \theta))^{-1/4} v_0 e_\theta:
        v_1
        \right]
        \right),
    \end{align*}
    where $u=(u_0,u_1), v=(v_0,v_1)$, and $e_\theta=(\cos \theta, \sin \theta) \in \R \cdot (\cos \theta, \sin \theta)=\mathcal{O}_{\RP^1}(-1)_{(\cos \theta, \sin \theta)}$.
    Because $e_{\theta+\pi}=-e_\theta$, this map satisfies $\Phi(\pi,[u,v])=\Phi(0,h([u,v]))$ for $h:T \rightarrow T$ given by $h([(u_0,u_1),(v_0,v_1)])=[(-u_0,u_1),(-v_0,v_1)]$.
    Using $(C \times C)/\{\pm 1\}_{\text{diag}} \simeq T^2$, we have that $\Phi$ descends to a map on the mapping torus
    \[
    \Phi: T^2\rtimes_h S^1 = ([0,\pi] \times T)/((\pi,z) \sim (0,h(z))) \rightarrow L
    \]
    that we denote by the same symbol.
    Under the identification $(C \times C)/\{\pm 1\}_{\text{diag}} \simeq T^2$, we have that $h$ acts as $h_*=-\operatorname{Id}:H_1(T^2,\R) \rightarrow H_1(T^2,\R)$ and $h_*=\operatorname{Id}:H_0(T^2,\R) \rightarrow H_0(T^2,\R)$.
    Computing $b^0$ and $b^1$ of this mapping torus is precisely \cite[Exercise 2.2.30(d)]{Hatcher2002}, and applying the long exact sequence for mapping tori from the exercise gives
    \[
    H_1(T^2,\R)
    \overset{\operatorname{Id}-h_*=2\operatorname{Id}}{\longrightarrow} H_1(T^2,\R)
    \rightarrow H_1(T^2\rtimes_h S^1,\R)
    \rightarrow H_0(T^2,\R)
    \overset{\operatorname{Id}-h_*=0}{\longrightarrow} H_0(T^2,\R),
    \]
    so $H_1(T^2\rtimes_h S^1,\R) \cong H_0(T^2,\R) \cong \R$.
    Also, $T^2\rtimes_h S^1$ is a fibration over $S^1$ with connected fibres, so $b^0(T^2\rtimes_h S^1)=1$.
    
    One checks that $\Phi$ is a diffeomorphism, hence $b^0(L)=1$ and $b^1(L)=1$.
    Hence, for the resolved $G_2$-manifold $N$, we find:
    $b^2(N)=2, b^3(N)=91$.
    These are the same Betti numbers as in Example \ref{example:cicy2}.
    It is plausible that the invariant $d_\pi$ from \cite[p.2]{Crowley2019} can be used to show that the two examples are not homeomorphic, but we were unable to prove this.
\end{ex}

\begin{remark}
    The computation of $h^{2,1}$ for some scrolls was explained in \cite[Remark 5]{Mboya2024}.
    However, the computation for the scroll $\mathbb{F}$ from the previous example was not carried out in the reference.
\end{remark}

\begin{ex}
    On $A=\P^1 \times \P^2$ with coordinates $([s:t],[x_0:x_1:x_2])$ define 
    \[
    q_0=s^4+t^4, \quad
    q_1=s^4+2t^4, \quad
    q_2=2s^4+t^4, \quad
    f=q_0 x_0^6+q_1 x_1^6 + q_2 x_2^6, \quad
    B=Z(f).
    \]
    Let $\rho: Y \rightarrow \P^1 \times \P^2$ be the double cover given as $Y=\{([s:t],x,w): f([s:t],x)=w^2\}$, where $([s:t],x,w) \sim ([\lambda s:\lambda t],\mu x, \lambda^2 \mu^3 w)$ for $\lambda,\mu \in \C^*$.
    Then, by \cite[Lemma 17.1(iii)]{Barth2003} we have $K_Y=\rho^*(K_{\P^1 \times \P^2} \otimes \mathcal{L}) \simeq \mathcal{O}_Y$, where $\mathcal{L}^{\otimes 2}=\mathcal{O}(B)=\mathcal{O}(4,6)$.
    Hence, $Y$ is Calabi-Yau.

    \emph{Real structure and fibration:}
    let $\eta=\mathcal{O}(4,6)$.
    One checks as in Example \ref{example:1} that $\eta$ satisfies conditions (a) and (b) from Proposition \ref{prop_lefschetz}.
    \emph{Non-empty real locus:}
    let $p \in \P^1 \times \P^2$ be real, then $f(p) \in \R$ and $(w,p)=(\pm \sqrt{f(p)}, p) \in Y$ are real points.
    \emph{Real discriminant locus:}
    for $[s:t]$, the fibre $Y_{[s:t]}$ is the K3 surface given as the branched double cover of $\P^2$ branched over a sextic.
    This is smooth if the branching curve is smooth \cite[p.3]{Huybrechts2016}.
    This is the case if the coefficients of $x_0^6$, $x_1^6$, and $x_2^6$ are all nonzero, which is the case for real points $[s:t]$.
    \emph{Smooth total space:}
    if $B$ is smooth, then $Y$ is smooth by the same argument as for the branched K3 surface under the previous point.
    We prove that $B$ is smooth.
    For a contradiction, assume $([s:t],[x_0:x_1:x_2])\in B$ is singular.
    If all $q_0,q_1,q_2 ([s:t]) \neq 0$, then the vanishing of the $x_0,x_1,x_2$ derivatives gives $x_0=x_1=x_2=0$, which is impossible.
    If exactly one $q_j$ for $j \in \{0,1,2\}$ vanishes at $[s:t]$, then the $x_0,x_1,x_2$ derivatives force $x_i=0$ for $i \neq j$, hence $x_j \neq 0$.
    Then the base derivative $d_{\mathbb{P}^1} f=(dq_j) x_j^6$, but $q_j$ has simple zeros, so $dq_j \neq 0$, which gives a contradiction.
    No two of the $q_i$ vanish at the same point, so this finishes the proof.

    Altogether, we have that (a)-(e) are satisfied for $B \subset \P^1 \times \P^2$.
    Thus, we can apply Proposition \ref{prop_lefschetz} to obtain a small perturbation of $B$ satisfying (1)-(5) from Proposition \ref{prop_lefschetz}, say $\tilde{B}$.
    By Lemma \ref{lefschetz_lemma_branched_cover}, we have that the branched cover over $\tilde{B}$ also has properties (1)-(5).
    As in the previous examples, Thus, Theorem \ref{thm:main1} applies, and from Corollary \ref{cor:joyce-karigiannis-for-cy} we obtain a $G_2$-manifold $N$.

    \emph{Computation of $G_2$-topology:}
    Write $\alpha=c_1(\mathcal{O}_{\P^1}(1))$ and $\beta=c_1(\mathcal{O}_{\P^2}(1))$.
    By \cite[Proposition 1.11 and Remark 1.12]{Lanteri1989}, we have $H^{1,1}(Y)=\C \rho^*\alpha \oplus \C \rho^* \beta$.
    Hence $\dim H^{1,1}(Y)^\sigma=0$ and $\dim H^{1,1}(Y)^{-\sigma}=2$, similar to the previous examples.

    We now compute $h^{2,1}(Y)$ with a calculation similar to the previous example.
    Taking the dual of \cite[Theorem II.8.17]{Hartshorne1977} gives the normal sequence
    \[
    0
    \rightarrow T_B
    \rightarrow T_A|_B
    \rightarrow N_{B/A}
    \rightarrow 0.
    \]
    Since $B \subset A$ is a smooth divisor, we have $N_{B/A} \simeq \mathcal{O}_A(B)|_B$.
    Taking Chern classes gives $c(TB)=\frac{c(TA)}{1+[B]} \big|_B$.
    Plugging in $c_1(A)=2\alpha+3\beta$ and $c_2(A)=6\alpha \beta+3\beta^2$ and $[B]=4\alpha+6\beta$, we get $c_2(TB)=(30\alpha \beta+21 \beta^2)|_B$, hence $\chi(B)=\int_B c_2(TB)=\int_A [B] (30\alpha \beta+21 \beta^2)=\int_A(84+180)\alpha \beta^2=264$, where in the last step we used $\alpha^2=\beta^3=0$.
    By \cite[Proposition 1.6]{Lanteri1989}, we have $\chi(Y)=2 \chi(A)-\chi(B)=12-264=-252$.
    Because $Y$ is Calabi-Yau, we have $\chi(Y)=2(h^{1,1}(Y)-h^{2,1}(Y))$, so $h^{2,1}(Y)=128$.

    For computing the Betti numbers of $L=\fix \sigma \subset Y$, take the diffeomorphism
    \begin{align*}
        \Phi : \RP^1 \times S^2 & \rightarrow L
        \\
        ([s:t],u) &\mapsto
        ([s:t],u,\sqrt{f(s,t,u)}),
    \end{align*}
    where the square root is the positive real square root.
    Hence, for $L=\fix(\sigma)$, we have $b^0(L)=1$ and $b^1(L)=1$, so for the resolved $G_2$-manifold $N$ we find
    $b^2(N)=2$ and $b^3(N)=133$.
    As in Example \ref{example:kaihnsa-10-s2-example}, these Betti numbers do not appear in the example lists in the literature and may be new.
\end{ex}

\bibliographystyle{abbrv}
\bibliography{oneforms-library}

\end{document}